\documentclass[11pt]{amsart}

\usepackage{amsmath,amssymb,amsthm,mathtools,bm}
\usepackage[margin=1in]{geometry}
\usepackage{hyperref}
\usepackage{algorithm}
\usepackage{algorithmic}
\usepackage{tcolorbox}

\theoremstyle{plain}
\newtheorem{theorem}{Theorem}
\newtheorem{lemma}{Lemma}

\theoremstyle{definition}
\newtheorem{defn}[theorem]{Definition}
\newtheorem*{notation}{Notation}
\newtheorem{assump}{Assumption}

\theoremstyle{remark}
\newtheorem*{remark}{Remark}

\newcommand{\dd}{\,\mathrm{d}}
\DeclarePairedDelimiter{\abs}{\lvert}{\rvert}

\usepackage{todonotes}

\usepackage{subcaption}

\title{Mirror Descent for Deterministic Optimal Control}
\date{\today} 
\author{Ye Feng}
\address{Mathematics Department, Duke University}
\email{ye.feng@duke.edu}

\author{Jianfeng Lu}
\address{Department of Mathematics, Department of Physics, Department of Chemistry, Duke University}
\email{jianfeng@math.duke.edu}

\begin{document}

\begin{abstract}
We study an explicit mirror-descent method for finite-horizon deterministic optimal control problems. The method is motivated by Pontryagin's maximum principle: at each iteration, one solves the state and adjoint equations and updates the control by maximizing a first-order approximation of the regularized Hamiltonian penalized by a Bregman divergence. In the Euclidean case, the update reduces to a projected gradient step in the control variable. Under global smoothness assumptions and uniform convexity of the mirror map, we prove a relative smoothness estimate for the cost functional and derive an energy dissipation inequality for sufficiently small step sizes. Under an additional concavity assumption on the unregularized Hamiltonian and convexity of the terminal cost, we establish relative convexity of the regularized objective. These estimates yield an \(O(1/n)\) convergence rate in the unregularized convex case and a geometric rate when the control regularization parameter is positive. Numerical examples illustrate the behavior of the method in linear-quadratic, degenerate convex, and nonlinear high-dimensional settings.
\end{abstract}

\maketitle

\section{Introduction}

Optimal control problems are a fundamental class of optimization problems with applications in engineering, economics, and scientific computing. A classical approach to such problems is based on Pontryagin’s maximum principle (PMP), which characterizes optimal controls through a system of forward and adjoint differential equations coupled with a pointwise optimality condition.

Beyond its role as a theoretical tool, PMP also naturally leads to iterative numerical methods. A prominent example is the method of successive approximations (MSA), which alternates between solving the state and adjoint equations and updating the control via a pointwise optimization of the Hamiltonian \cite{krylov1972algorithm,chernousko1982method}. In practice, this update step is often implemented either as a projected gradient-type step or through implicit schemes that solve a regularized local problem (for example, see \cite{sethi2024modified}). These methods motivate the study of explicit first-order control updates and their convergence properties.

In this work, we study an explicit control update scheme derived from PMP and give a convergence analysis using a Bregman divergence framework. We consider a regularized optimal control problem of the form
\[
J^\tau(u) = \int_0^T \big(f_t(x_t,u_t) + \tau h(u_t)\big)\dd t + g(x_T),
\]
and analyze an iterative scheme in which the control is updated by maximizing a first-order approximation of the regularized Hamiltonian penalized by a Bregman proximal term. When the regularizer $h(u) = \tfrac{1}{2}|u|^2$, this update reduces to an explicit projected gradient step in the control variable. The resulting algorithm is fully explicit and is thus directly comparable to standard first-order methods.

The use of a general convex regularizer allows us to formulate the algorithm in terms of Bregman divergences, which provides a convenient and flexible framework for the analysis. Although our primary interest is in the Euclidean setting, this formulation enables a unified treatment of the update step and simplifies the derivation of key estimates.

Our work is related to several strands of literature. On the one hand, there has been growing interest in mirror descent and Bregman proximal methods in optimal control and related infinite-dimensional optimization problems, including stochastic or measure-valued control settings, as well as continuous-time limits formulated as gradient flows \cite{reisinger2023linear,kerimkulov2025mirror,sethi2025mirror}. On the other hand, recent works have studied variants of MSA with implicit or regularized updates \cite{kerimkulov2021modified,sethi2024modified}. In contrast, the present paper focuses on a deterministic control setting and provides a direct analysis of a fully explicit iterative scheme derived from PMP.

Following a variational perspective similar to that of \cite{kerimkulov2025mirror}, we base our analysis on the first variation formula, which expresses the derivative of the cost functional in terms of the gradient of the Hamiltonian. This allows us to interpret the update step as a mirror-descent-type iteration. We show that the cost functional satisfies a relative smoothness property with respect to the induced Bregman divergence, leading to an energy dissipation inequality along the iterates. Under an additional convexity assumption on the Hamiltonian, we further establish a relative convexity property with modulus given by the regularization parameter. Combining these properties, we derive convergence rates for the proposed algorithm: a sublinear rate of order $O(1/n)$ in the unregularized case and a geometric rate in the presence of regularization. Together, these results provide a unified variational framework for analyzing explicit first-order PMP-based methods in deterministic optimal control.

The remainder of the paper is organized as follows. In Section~\ref{sec:formulation}, we introduce the problem formulation and standing assumptions. Section~\ref{sec:max_prin} reviews Pontryagin’s maximum principle for deterministic controls and establishes well-posedness of the state and adjoint equations. In Section~\ref{sec:first_var}, we derive the first variation formula for the cost functional, introduce the Bregman divergence, and formulate the mirror descent-based MSA algorithm. Section~\ref{sec:re_smooth} establishes relative smoothness, and Section~\ref{sec:energy} proves an energy dissipation property. Section~\ref{sec:re_convex} establishes relative convexity under an additional structural assumption. Finally, Section~\ref{sec:conv_rates} combines the above ingredients to obtain convergence rates. Section~\ref{sec:numerics} includes three numerical examples: a linear-quadratic example with an explicit solution as a sanity check, a nonquadratic example highlighting the qualitative difference between the regularized and unregularized cases, and a high-dimensional nonlinear example used to study dimension dependence.

\subsection*{Related work}

Iterative methods for optimal control derived from Pontryagin's maximum principle (PMP) have been studied extensively. A classical example is the method of successive approximations (MSA), and several recent works have considered modified variants that stabilize the control update through augmented or proximal-type regularization. In particular, \cite{kerimkulov2021modified} proposed a modified MSA for stochastic control problems with controls entering both the drift and diffusion, and established convergence under general assumptions together with convergence rates under additional structural conditions. More recently, \cite{sethi2024modified} showed that a modified MSA can be interpreted as an implicit Euler discretization of a gradient-flow system, and studied both the convergence of the interpolated iterates to the corresponding flow and their asymptotic behavior. Related proximal policy-gradient methods for finite-horizon continuous-time stochastic control were studied in \cite{reisinger2023linear}, where linear convergence to stationary points was established under suitable assumptions.

Proximal and trust-region ideas also play an important role in reinforcement learning and policy optimization. Representative examples include trust region policy optimization (TRPO) \cite{schulman2015trust} and proximal policy optimization (PPO) \cite{schulman2017proximal}, which stabilize policy updates by controlling the deviation from the previous iterate, typically through KL-divergence or related proximal penalties. Although these methods arise in a different algorithmic setting, they provide a broader connection between stabilized policy updates and proximal optimization ideas.

A closely related direction studies mirror-descent and Bregman-geometric formulations of control and policy optimization. In \cite{kerimkulov2025mirror}, mirror descent was developed for stochastic control problems with measure-valued controls, and relative smoothness and relative convexity of the objective were established with respect to the Bregman divergence induced by the regularizer, leading to convergence-rate estimates. Related developments can be found in \cite{sethi2025mirror,sethi2025entropy}. In particular, \cite{sethi2025mirror} considers mirror-descent-type methods for constrained stochastic control problems with Markov controls, while \cite{sethi2025entropy} develops a policy mirror-descent framework for continuous-time stochastic control with entropy regularization and annealing.

Mirror-descent ideas have also appeared in model predictive control (MPC). Wagener et al.~\cite{wagener2019online} formulated MPC as an online learning problem and proposed a dynamic mirror-descent framework for updating finite-dimensional parameters of a receding-horizon controller. In a different direction, Bregman distances have been used in deterministic optimal control as regularization tools for constrained optimization problems. The iterative Bregman regularization method of \cite{porner2016iterative} provides convergence and error estimates for problems with inequality constraints, and \cite{porner2018inexact} extends this framework to inexact settings with discretization and numerical errors.

The present paper considers a finite-horizon deterministic optimal control problem with open-loop controls and studies a mirror-descent iteration directly in the control trajectory. The update is derived from the PMP first-variation formula, and the analysis is based on relative smoothness and relative convexity with respect to the Bregman divergence. The main purpose is to give a direct and self-contained formulation of this explicit PMP-based mirror-descent scheme and its basic convergence properties.

\begin{notation}
    For a vector $x = (x_1, \ldots, x_n) \in \mathbb{R}^n$ and a matrix $A = (a_{ij}) \in \mathbb{R}^{m \times n}$, $\abs{x}$ denotes the Euclidean norm of $x$ and $\abs{A}$ denotes the spectral norm of $A$. Throughout the paper, $\|\cdot\|$ denotes the operator norm.
\end{notation}

\section{Problem Formulation}\label{sec:formulation}

We consider a deterministic control problem on a finite horizon $T > 0$. The state $x_t\in\mathbb R^d$ evolves according to
\begin{equation}\label{eq:state}
\dot x_t = b(t,x_t,u_t)
\end{equation}
with fixed initial condition $x_0\in\mathbb R^d$. Here $b:[0,T]\times\mathbb R^d\times U\to\mathbb R^d$, and the control $u$ belongs to $\mathcal U:=C([0,T];U)$, where $U$ is a closed convex subset of $\mathbb R^m$. We often write $b_t(x,u):=b(t,x,u)$ and similarly for $f_t$.

We define the regularized cost functional
\begin{equation}\label{eq:cost}
J^\tau(u) := \int_0^T \bigl(f_t(x_t,u_t)+\tau h(u_t)\bigr)\dd t + g(x_T),
\end{equation}
with running cost $f:[0,T]\times\mathbb R^d\times U\to\mathbb R$, terminal cost $g:\mathbb R^d\to\mathbb R$, convex regularizer $h:U\to\mathbb R$, and regularization parameter $\tau\ge 0$.
It is convenient to decompose
\[
J^\tau(u)=J^0(u)+\tau H(u),
\qquad
H(u):=\int_0^T h(u_t)\dd t.
\]
A pair $(u, x)$ is \textit{admissible} if $u\in \mathcal{U}$ and $x$ is the unique solution of \eqref{eq:state}. An admissible pair $(u^*, x^*)$ is \textit{optimal} if
\begin{align}
    J^{\tau}(u^*) \leq J^{\tau}(u) \nonumber
\end{align}
for every admissible pair $(u, x)$. Formally, the control problem is to find an optimal control $u^{*}$ by solving the minimization problem
\begin{align}
    (\mathbf{P}) \quad
    \inf_{u \in \mathcal{U}} J^\tau(u). \nonumber
\end{align}

\begin{assump}\label{asp:main_asp}
    We make the following standing assumptions:
    \begin{enumerate}
        \item $U$ is a nonempty, closed, convex subset of $\mathbb{R}^m$.
        \item The maps $b: [0,T] \times \mathbb{R}^d \times U \to \mathbb{R}^d$ and $f: [0,T] \times \mathbb{R}^d \times U\to \mathbb{R}$ are continuous, and there exists a constant $M > 0$ such that for $\varphi = b, f$,
              \begin{align}
                  |\varphi_t(x, u) - \varphi_t(x', u')| \leq M(|x - x'| + |u - u'|), \quad |\varphi_t(0, u)| \leq M, \nonumber
              \end{align}
              for any $t \in [0,T]$, $x, x' \in \mathbb{R}^d$, and $u, u' \in U$.
        \item For $\varphi = b, f$, the map $\varphi_t(x,u)$ is twice continuously differentiable in $(x,u)$, and, possibly after increasing $M$, the same constant satisfies
              \begin{align}
                  |\nabla_x \varphi_t(x, u) - \nabla_x \varphi_t(x', u')| \leq M(|x - x'| + |u - u'|), \nonumber
              \end{align}
              for any $t \in [0,T]$, $x, x' \in \mathbb{R}^d$, and $u, u' \in U$.
              Further, assume that there exist constants $M_{b,uu}, M_{f,uu} > 0$ such that
\[
\|\nabla^2_{uu} b_t(x,u)\| \le M_{b,uu},
\qquad
\|\nabla^2_{uu} f_t(x,u)\| \le M_{f,uu},
\]
for all $(t,x,u) \in [0,T]\times\mathbb{R}^d\times U$.
        \item The map $g: \mathbb{R}^d \to \mathbb{R}$ is continuously differentiable and satisfies
        \[
        \max\{|g(x) - g(x')|, |\nabla g(x) - \nabla g(x')|\} \leq M|x - x'|, \quad |g(0)| \leq M,
        \]
        for any $x, x' \in \mathbb{R}^d$.
        \item The regularizer $h : U \to \mathbb{R}$ is convex and continuously differentiable.
        \item There exists a (continuous) optimal solution to Problem (\textbf{P}).
    \end{enumerate}
\end{assump}

\begin{remark}
    Assumption~\ref{asp:main_asp} (2) and (3) imply that, on $[0,T]\times\mathbb{R}^d\times U$, the first-order derivatives
    \[
    \nabla_x b_t(x,u),\quad \nabla_u b_t(x,u),\quad
    \nabla_x f_t(x,u),\quad \nabla_u f_t(x,u),
    \]
    and the second-order derivatives
    \[
    \nabla^2_{xx} b_t(x,u),\quad \nabla^2_{xu} b_t(x,u),\quad
    \nabla^2_{xx} f_t(x,u),\quad \nabla^2_{xu} f_t(x,u)
    \]
    are uniformly bounded by $M$. 
\end{remark}

\section{Pontryagin's Maximum Principle} \label{sec:max_prin}

In this section, we recall a necessary condition for optimality in Problem (\textbf{P}), namely Pontryagin's maximum principle. We also establish well-posedness of the state and adjoint equations for continuous controls, together with a uniform boundedness estimate.

Define the unregularized Hamiltonian
\begin{equation}\label{eq:H0}
\mathcal H_t^0(x,p,u):= p\cdot b_t(x,u)- f_t(x,u),
\end{equation}
and the regularized Hamiltonian
\begin{equation}\label{eq:Htau}
\mathcal H_t^\tau(x,p,u):= \mathcal H_t^0(x,p,u) - \tau h(u).
\end{equation}

Given $u\in\mathcal U$ and the corresponding state $x$, we define the adjoint $p$ by
\begin{equation}\label{eq:adjoint}
\dot p_t = - \nabla_x \mathcal H_t^0(x_t,p_t,u_t),
\qquad
p_T = - \nabla g(x_T).
\end{equation}
Equivalently, one may use $\nabla_x\mathcal H_t^\tau$ in \eqref{eq:adjoint}, since the regularizer $h$ is independent of the state variable.

\begin{lemma}\label{lmm:a_priori}
    Under Assumption~\ref{asp:main_asp}, for any control $u \in \mathcal{U} = C([0, T]; U)$, the state equation \eqref{eq:state} admits a unique solution $x$, and the adjoint equation \eqref{eq:adjoint} admits a unique solution $p$. Moreover, $|x_t| \leq M_X$ and $|p_t| \leq M_P$ for all $t \in [0,T]$ with constants given by $M_X := (|x_0|+MT)e^{MT}$ and $M_P := M(1+T)e^{MT}$.
\end{lemma}

\begin{proof}
    Let $u \in \mathcal U$, under Assumption~\ref{asp:main_asp}, the state equation \eqref{eq:state} has a unique solution $x$ for the control $u$. Since $x$ is continuous, Assumption~\ref{asp:main_asp} also yields a unique solution $p$ of the adjoint equation \eqref{eq:adjoint} associated with $(u,x)$. To establish the boundedness estimate, observe that $x$ satisfies the integral form of the forward equation,
    \begin{align}
        x_t = x_0 + \int_{0}^{t} b_s(x_s, u_s )\dd s , \quad \forall t \in [0,T]. \nonumber
    \end{align}
    Taking norms and applying Assumption~\ref{asp:main_asp} (2),
    \begin{align}
        |x_t| & \leq |x_0| + \int_{0}^{t} |b_s(x_s, u_s)| \dd s, \nonumber \\
        & \leq |x_0| + \int_{0}^{t} \left( M|x_s| + M \right) \dd s,
        \qquad \forall t \in [0,T]. \nonumber
    \end{align}
    By Gr\"onwall's inequality,
    \begin{align}\label{eq:bdd_x}
        |x_t| & \leq ( |x_0| + Mt) e^{Mt} \leq ( |x_0| + MT) e^{MT} = M_X, \quad \forall t \in [0,T].
    \end{align}
    Similarly, $p$ satisfies the integral form of the adjoint equation,
    \begin{align}
        p_t = - \nabla_x g(x_T) +  \int^{T}_{t} \nabla_x b_s( x_s, u_s)^\top p_s \dd s - \int_{t}^{T} \nabla_x f_s( x_s, u_s) \dd s, \quad \forall t \in [0,T]. \nonumber
    \end{align}
    Taking norms and applying Assumption~\ref{asp:main_asp} together with \eqref{eq:bdd_x},
    \begin{align}
        |p_t| & \leq |\nabla_x g(x_T)| + \int_{t }^{T} |\nabla_x b_s(x_s, u_s)| \cdot |p_s| \dd s + \int_{t }^{T} |\nabla_x f_s(x_s, u_s)| \dd s \nonumber \\
        & \leq M + M \int_{t }^{T} |p_s| \dd s + M(T-t), \quad \forall t \in [0,T]. \nonumber
    \end{align}
    By Gr\"onwall's inequality,
    \begin{align}
        |p_t| & \leq M (1 + T-t) \cdot e^{M(T-t)} \leq M(1+T)e^{MT} = M_P, \quad \forall t \in [0,T]. \nonumber \qedhere
    \end{align}
\end{proof}

\smallskip 

We now state the necessary condition for being an optimal control.

\begin{theorem}[Pontryagin's maximum principle] \label{thm:cont_max_prin}
    Let $(u^*,x^*)$ be an optimal pair for Problem (\textbf{P}). Under Assumption~\ref{asp:main_asp}, there exists $p^*: [0,T] \to \mathbb{R}^d$ satisfying the adjoint equation \eqref{eq:adjoint} associated with $(u^*, x^*)$ and
    \begin{align} \label{eq:max_cond_gen}
                  \mathcal{H}_t^\tau(x^*_t, p^*_t, u^*_t) = \max_{u \in U} \mathcal{H}_t^\tau(x^*_t, p^*_t, u), \quad \text{for a.e. } t \in [0,T].
    \end{align}
\end{theorem}

The condition \eqref{eq:max_cond_gen} is called the \textit{maximum condition}. The proof of Theorem~\ref{thm:cont_max_prin} is standard and hence omitted, see e.g., \cite[Chapter 3]{yong1999stochastic}.

\section{First Variation and Bregman Mirror Descent}\label{sec:first_var}

\subsection{First variation formula}

We now derive a first variation formula for $J^{\tau}$ that expresses its derivative in terms of the gradient of the Hamiltonian with respect to the control variable.

\begin{lemma}[First variation formula]\label{lem:first-variation}
Under Assumption~\ref{asp:main_asp}, for every $u,v\in\mathcal U$,
the one-sided directional derivative of $J^0$ at $u$ along the feasible segment from $u$ to $v$ exists and is given by
\begin{equation}\label{eq:first-var-J0}
dJ^0(u)(v-u)
=
-\int_0^T \nabla_u \mathcal H_t^0(x_t,p_t,u_t)\cdot (v_t-u_t)\dd t,
\end{equation}
where $(x, p)$ is the state-adjoint pair corresponding to $u$.

Moreover, the corresponding one-sided directional derivative of $J^\tau$ is
\begin{equation}\label{eq:first-var-Jtau}
dJ^\tau(u)(v-u)
=
- \int_0^T \nabla_u \mathcal H_t^\tau(x_t,p_t,u_t)\cdot (v_t-u_t)\dd t,
\end{equation}
where
\[
\nabla_u \mathcal H_t^\tau(x,p,u)
=
\nabla_u \mathcal H_t^0(x,p,u) - \tau \nabla h(u).
\]
\end{lemma}

The proof of Lemma~\ref{lem:first-variation} is deferred to Appendix~\ref{sec:prf}.

\subsection{Bregman divergence and the mirror descent algorithm}\label{sec:div_alg}

We next introduce the Bregman divergence induced by the convex regularizer and formulate a mirror descent algorithm, based on Pontryagin's maximum principle, for approximating the optimal control.

\begin{defn}
Assume $h:U\to\mathbb R$ is convex and differentiable. The Bregman divergence induced by $h$ is
\begin{equation}\label{eq:bregman-pointwise}
D_h(v\mid u)
:=
h(v)-h(u)-\nabla h(u)\cdot (v-u),
\qquad u,v\in U.
\end{equation}
The corresponding integrated version on the control space is
\begin{equation}\label{eq:bregman-path}
\mathcal{D}_h(v\mid u)
:=
\int_0^T D_h(v_t\mid u_t)\dd t, \qquad u, v \in \mathcal{U}.
\end{equation}
\end{defn}

To ensure sufficient curvature of the mirror map and enable descent estimates for the algorithm, we additionally assume that $h$ is uniformly convex.

\begin{assump}[Uniform convexity of the mirror map]\label{ass:h-unif-conv}
Assume that $h:U\to \mathbb R$ is continuously differentiable and there exists $\sigma_h>0$ such that
for all $u,v\in U$,
\begin{equation}\label{eq:h-unif-conv}
D_h(v\mid u)\ge \frac{\sigma_h}{2}|v-u|^2.
\end{equation}
Consequently,
\begin{equation}\label{eq:H-unif-conv}
\mathcal{D}_h(v\mid u)\ge \frac{\sigma_h}{2}\|v-u\|_{L^2}^2,
\qquad \forall u,v\in \mathcal U.
\end{equation}
\end{assump}

Algorithm~\ref{alg:mirror-pmp} gives the mirror descent method for optimal control. The method can be viewed as a mirror-descent variant of the classical method of successive approximations.

\begin{algorithm}[H]
\caption{Mirror Descent Algorithm for Optimal Control}
\label{alg:mirror-pmp}
\begin{algorithmic}[1]

\STATE \textbf{Input:} $\lambda > 0$, initial guess $u^0 \in \mathcal U$

\FOR{$n = 0,1,2,\dots$}

    \STATE \textbf{Step 1 (State equation):} Solve
    \[
    \dot x_t^n = b_t(x_t^n,u_t^n),
    \qquad x_0^n = x_0.
    \]

    \STATE \textbf{Step 2 (Adjoint equation):} Solve
    \[
    \dot p_t^n = -\nabla_x \mathcal H_t^0(x_t^n,p_t^n,u_t^n),
    \qquad
    p_T^n = - \nabla g(x_T^n).
    \]

    \STATE \textbf{Step 3 (Mirror update):} For every $t \in [0,T]$, compute
    \begin{align}\label{eq:mirror-update}
        u_t^{n+1}
    & \in
    \arg\max_{v\in U}
    \Big\{
    \nabla_u \mathcal H_t^\tau(x_t^n,p_t^n,u_t^n)\cdot (v-u_t^n)
    -
    \lambda D_h(v\mid u_t^n)
    \Big\} \\
    & \label{eq:mirror-update-expanded}
    =
    \arg\max_{v\in U}
    \Big\{
    \nabla_u \mathcal H_t^0(x_t^n,p_t^n,u_t^n)\cdot (v-u_t^n)
    -
    \tau \nabla h(u_t^n)\cdot (v-u_t^n)
    -
    \lambda D_h(v\mid u_t^n)
    \Big\}.
    \end{align}

\ENDFOR

\end{algorithmic}
\end{algorithm}

\begin{remark}
   Euclidean projected gradient descent is a special case of mirror descent. To see this, take
\[
h(u)=\frac12 |u|^2.
\]
Then
\[
D_h(v\mid u)=\frac12|v-u|^2,
\qquad
\nabla h(u)=u.
\]
Hence the mirror step \eqref{eq:mirror-update-expanded} becomes
\begin{equation}\label{eq:euclidean-prox}
u_t^{n+1}
\in
\arg\max_{v\in U}
\Big\{
\nabla_u \mathcal H_t^0(x_t^n,p_t^n,u_t^n)\cdot (v-u_t^n)
-
\tau u_t^n\cdot (v-u_t^n)
-
\frac{\lambda}{2}|v-u_t^n|^2
\Big\}.
\end{equation}
Equivalently,
\begin{equation}\label{eq:proj-gradient}
u_t^{n+1}
=
\Pi_U\left(
u_t^n + \frac{1}{\lambda}\nabla_u \mathcal H_t^\tau(x_t^n,p_t^n,u_t^n)
\right).
\end{equation}
If $U=\mathbb R^m$, then
\begin{equation}\label{eq:plain-gradient}
u_t^{n+1}
=
\left(1-\frac{\tau}{\lambda}\right)u_t^n
+\frac{1}{\lambda}\nabla_u \mathcal H_t^0(x_t^n,p_t^n,u_t^n).
\end{equation}
Thus Euclidean projected gradient descent is the special case of mirror descent associated with the quadratic mirror map $h(u)=\tfrac12|u|^2$.
\end{remark}

\begin{lemma}[Preservation of admissibility]
Assume that Assumptions~\ref{asp:main_asp} and~\ref{ass:h-unif-conv} hold. Let $\{u^n\}_{n\geq 0}$ be the control sequence generated by Algorithm~\ref{alg:mirror-pmp}. If $u^0 \in \mathcal U$, then $u^n \in \mathcal U$ for all $n \ge 0$.
\end{lemma}

\begin{proof}
Suppose that $u^n \in \mathcal{U}$. By definition of the mirror step \eqref{eq:mirror-update}, for each $t \in [0,T]$ we have
$u_t^{n+1} \in U$.

Set
\[
a_t := \nabla_u \mathcal H_t^\tau(x_t^n, p_t^n, u_t^n),
\qquad
\eta_t := a_t + \lambda \nabla h(u_t^n).
\]
Then \eqref{eq:mirror-update} is equivalent to
\[
u_t^{n+1} \in \arg\max_{v \in U} \bigl\{  \eta_t \cdot v - \lambda h(v) \bigr\}.
\]
By Assumption~\ref{ass:h-unif-conv}, $h$ is $\sigma_h$-strongly convex. The same lower quadratic bound on $h$ makes $\lambda h(v)-\eta_t\cdot v$ coercive, so the maximizer exists even when $U$ is unbounded; strong convexity gives uniqueness.
It satisfies the variational inequality
\[
\bigl( \eta_t - \lambda \nabla h(u_t^{n+1}) \bigr)\cdot (w - u_t^{n+1}) \le 0,
\qquad \forall\, w \in U.
\]

Let $s,t \in [0,T]$. Testing the above inequality at time $t$ with $w=u_s^{n+1}$ and at time $s$ with $w=u_t^{n+1}$, then adding the two inequalities, we obtain
\[
\lambda \bigl(\nabla h(u_t^{n+1}) - \nabla h(u_s^{n+1})\bigr)
\cdot (u_t^{n+1} - u_s^{n+1})
\le
(\eta_t - \eta_s)\cdot (u_t^{n+1} - u_s^{n+1}).
\]
Assumption~\ref{ass:h-unif-conv} implies
\[
(\nabla h(u_t^{n+1})-\nabla h(u_s^{n+1}))\cdot
(u_t^{n+1}-u_s^{n+1})
\ge \sigma_h |u_t^{n+1}-u_s^{n+1}|^2.
\]
Hence
\[
\lambda\sigma_h |u_t^{n+1}-u_s^{n+1}|^2
\le (\eta_t-\eta_s)\cdot (u_t^{n+1}-u_s^{n+1}).
\]
By the Cauchy--Schwarz inequality,
\[
\lambda\sigma_h |u_t^{n+1}-u_s^{n+1}|^2
\le |\eta_t-\eta_s|\,|u_t^{n+1}-u_s^{n+1}|,
\]
hence
\[
|u_t^{n+1} - u_s^{n+1}|
\le
\frac{1}{\lambda \sigma_h} |\eta_t - \eta_s|.
\]

Since $u^n$, $x^n$, and $p^n$ are continuous and $\nabla_u \mathcal H_t^\tau(x,p,u)$ is continuous under Assumption~\ref{asp:main_asp}, the map $t \mapsto \eta_t$ is continuous. Therefore $t \mapsto u_t^{n+1}$ is continuous on $[0,T]$, so $u^{n+1} \in C([0,T];U) = \mathcal U$. The conclusion follows by induction on $n$.
\end{proof}

In the following sections, we analyze the convergence properties of the mirror descent algorithm. Although the maximum principle motivates the update, the convergence analysis relies only on the first-variation representation of $J$ in terms of the Hamiltonian gradient, as in Lemma~\ref{lem:first-variation}.

\section{Relative Smoothness}\label{sec:re_smooth}

This section shows that, under Assumption~\ref{ass:h-unif-conv}, the cost functional satisfies a relative smoothness estimate derived from the first variation formula above.

\begin{lemma}[Stability of the state and adjoint]\label{lem:state-adjoint-stability}
Under Assumption~\ref{asp:main_asp}, there exist constants $C_x, C_p >0$ such that for all $u,v\in\mathcal U$, if $x^u,x^v$ are the corresponding state trajectories and $p^u,p^v$ are the corresponding adjoints, then
\begin{equation}\label{eq:state-stability}
\|x^v-x^u\|_{L^2}
\le
C_x \,\|v-u\|_{L^2},
\end{equation}
and
\begin{equation}\label{eq:adjoint-stability}
\|p^v-p^u\|_{L^2}
\le
C_p \,\|v-u\|_{L^2}.
\end{equation}
\end{lemma}

\begin{proof}
Fix $u,v\in\mathcal U$ and write
\[
\delta u := v-u,
\qquad
\delta x_t := x_t^v-x_t^u,
\qquad
\delta p_t := p_t^v-p_t^u.
\]

\smallskip

\noindent
\textit{Step 1: Estimate for the state.}
By the state equation,
\[
\dot{\delta x}_t
=
b_t(x_t^v,v_t)-b_t(x_t^u,u_t).
\]
Hence, by Assumption~\ref{asp:main_asp} (2),
\[
|\delta x_t|
\le
\int_0^t |b_s(x_s^v,v_s)-b_s(x_s^u,u_s)|\dd s
\le
M\int_0^t \bigl(|\delta x_s|+|\delta u_s|\bigr)\dd s.
\]
By Gr\"onwall's inequality,
\[
\sup_{t \in [0,T]} |\delta x_t|
\le
M e^{MT}\|\delta u\|_{L^1}
\le M \sqrt{T} e^{MT} \|\delta u\|_{L^2}.
\]
Hence
\[
\| \delta x\|_{L^2} \leq \sqrt{T} \sup_{t \in [0,T]} |\delta x_t| \leq MT e^{MT} \| \delta u\|_{L^2}.
\]
Defining $C_x := M T e^{MT} $ proves \eqref{eq:state-stability}.

\smallskip

\noindent
\textit{Step 2: Estimate for the adjoint.}
Recall that
\[
\dot p_t^u
=
-\nabla_x b_t(x_t^u,u_t)^\top p_t^u + \nabla_x f_t(x_t^u,u_t),
\qquad
p_T^u= - \nabla g(x_T^u),
\]
and similarly for $p^v$. Therefore
\[
\begin{aligned}
\dot{\delta p}_t
&=
-\nabla_x b_t(x_t^v,v_t)^\top p_t^v+\nabla_x b_t(x_t^u,u_t)^\top p_t^u
+ \bigl(\nabla_x f_t(x_t^v,v_t)-\nabla_x f_t(x_t^u,u_t)\bigr) \\
&=
-\nabla_x b_t(x_t^v,v_t)^\top \delta p_t
-\bigl(\nabla_x b_t(x_t^v,v_t)-\nabla_x b_t(x_t^u,u_t)\bigr)^\top p_t^u \\
&\quad
+\bigl(\nabla_x f_t(x_t^v,v_t)-\nabla_x f_t(x_t^u,u_t)\bigr).
\end{aligned}
\]
Also,
\[
\delta p_T
=
- (\nabla g(x_T^v)-\nabla g(x_T^u)).
\]
Hence, using Assumption~\ref{asp:main_asp} (2)--(4) and Lemma~\ref{lmm:a_priori}, for every $t\in[0,T]$,
\[
\begin{aligned}
|\delta p_t|
&\le
|\delta p_T|
+
\int_t^T |\nabla_x b_s(x_s^v,v_s)|\,|\delta p_s|\,\dd s \\
&\quad
+
\int_t^T
\bigl|
\nabla_x b_s(x_s^v,v_s)-\nabla_x b_s(x_s^u,u_s)
\bigr|\,|p_s^u|\,\dd s
+
\int_t^T
\bigl|
\nabla_x f_s(x_s^v,v_s)-\nabla_x f_s(x_s^u,u_s)
\bigr|\,\dd s \\
&\le
M|\delta x_T|
+
M\int_t^T |\delta p_s|\,\dd s
+
M(M_P+1)\int_t^T \bigl(|\delta x_s|+|\delta u_s|\bigr)\,\dd s.
\end{aligned}
\]
Using \textit{Step 1} and Gr\"onwall's inequality, we obtain
\[
\| \delta p\|_{L^2} \leq \sqrt{T} \sup_{t \in [0,T]} |\delta p_t|
\le
e^{MT} M \Big(C_x + (M_P+1) (C_x +1)T  \Big) \|\delta u\|_{L^2} =: C_p \|\delta u\|_{L^2}.
\]
This proves \eqref{eq:adjoint-stability}.
\end{proof}

\begin{lemma}[Relative smoothness]\label{lem:relative-smoothness}
Under Assumptions~\ref{asp:main_asp} and~\ref{ass:h-unif-conv}, there exists a positive constant
$L$ such that for all $u,v\in\mathcal U$,
\begin{equation}\label{eq:relative-smoothness}
J^\tau(v)-J^\tau(u)
\le
dJ^\tau(u)(v-u)+L\,\mathcal{D}_h(v\mid u).
\end{equation}
Equivalently, by Lemma~\ref{lem:first-variation},
\begin{equation}\label{eq:relative-smoothness-H}
J^\tau(v)-J^\tau(u)
\le
-\int_0^T \nabla_u \mathcal H_t^\tau(x_t^u,p_t^u,u_t)\cdot (v_t-u_t)\dd t
+
L\int_0^T D_h(v_t\mid u_t)\dd t.
\end{equation}
\end{lemma}

\begin{proof}
Fix $u,v\in\mathcal U$ and set
\[
\delta u := v-u.
\]
For each $\theta\in[0,1]$, define
\[
u^\theta := u+\theta \delta u.
\]
Let $x^\theta$ be the state corresponding to $u^\theta$, and let $p^\theta$ be the adjoint associated with
$(x^\theta,u^\theta)$.

\smallskip

\noindent
\textit{Step 1: Fundamental theorem of calculus along the segment.}
Define
\[
\Phi(\theta):=J^\tau(u^\theta), \qquad \theta\in[0,1].
\]
For $\theta\in[0,1)$, Lemma~\ref{lem:first-variation} applied at the base point $u^\theta$ in the direction $\delta u$ gives
\[
\Phi'_+(\theta)=dJ^\tau(u^\theta)(\delta u)
=
-\int_0^T
\nabla_u \mathcal H_t^\tau(x_t^\theta,p_t^\theta,u_t^\theta)\cdot \delta u_t\dd t.
\]
The right-hand side is continuous in $\theta$, so the fundamental theorem of calculus along the segment yields
\begin{equation}\label{eq:FTC-Jtau}
J^\tau(v)-J^\tau(u)
=
\int_0^1 \Phi'_+(\theta)\dd\theta
=
-\int_0^1\int_0^T
\nabla_u \mathcal H_t^\tau(x_t^\theta,p_t^\theta,u_t^\theta)\cdot \delta u_t\dd t\dd\theta.
\end{equation}

Subtracting \eqref{eq:first-var-Jtau} from \eqref{eq:FTC-Jtau} gives
\begin{equation}\label{eq:decomp-J}
\begin{aligned}
J^\tau(v)-J^\tau(u)-dJ^\tau(u)(v-u)
&=
-\int_0^1\int_0^T
\Bigl(
\nabla_u \mathcal H_t^\tau(x_t^\theta,p_t^\theta,u_t^\theta)
-
\nabla_u \mathcal H_t^\tau(x_t^u,p_t^u,u_t)
\Bigr)\cdot \delta u_t\dd t\dd \theta.
\end{aligned}
\end{equation}

\smallskip

\noindent
\textit{Step 2: Split the regularized and unregularized parts.}
Since
\[
\nabla_u \mathcal H_t^\tau(x,p,u)=\nabla_u \mathcal H_t^0(x,p,u) - \tau \nabla h(u),
\]
the right-hand side of \eqref{eq:decomp-J} equals
\[
I_1+I_2,
\]
where
\[
I_1:=
-\int_0^1\int_0^T
\Bigl(
\nabla_u \mathcal H_t^0(x_t^\theta,p_t^\theta,u_t^\theta)
-
\nabla_u \mathcal H_t^0(x_t^u,p_t^u,u_t)
\Bigr)\cdot \delta u_t\dd t\dd \theta,
\]
and
\[
I_2:=
\tau\int_0^1\int_0^T
\bigl(\nabla h(u_t^\theta)-\nabla h(u_t)\bigr)\cdot \delta u_t\dd t\dd \theta.
\]

By the fundamental theorem of calculus applied pointwise to $h$,
\[
h(v_t)-h(u_t)-\nabla h(u_t)\cdot (v_t-u_t)
=
\int_0^1
\bigl(\nabla h(u_t^\theta)-\nabla h(u_t)\bigr)\cdot \delta u_t\dd\theta.
\]
Integrating in time yields
\[
I_2
=
\tau \mathcal{D}_h(v\mid u).
\]

\smallskip

\noindent
\textit{Step 3: Estimate the unregularized remainder $I_1$.}
We first write
\[
\nabla_u \mathcal H_t^0(x,p,u)
=
\nabla_u b_t(x,u)^\top p - \nabla_u f_t(x,u).
\]
Fix $t\in[0,T]$ and two triples $(x,p,u)$, $(\tilde x,\tilde p,\tilde u)$ satisfying
\[
|x|,|\tilde x|\le M_X,
\qquad
|p|,|\tilde p|\le M_P.
\]
Then
\[
\begin{aligned}
&\nabla_u \mathcal H_t^0(x,p,u)-\nabla_u \mathcal H_t^0(\tilde x,\tilde p,\tilde u) \\
&=
\bigl(\nabla_u b_t(x,u)-\nabla_u b_t(\tilde x,\tilde u)\bigr)^\top p
+\nabla_u b_t(\tilde x,\tilde u)^\top (p-\tilde p)
-\bigl(\nabla_u f_t(x,u)-\nabla_u f_t(\tilde x,\tilde u)\bigr).
\end{aligned}
\]
Hence,
\[
\begin{aligned}
&\bigl|
\nabla_u \mathcal H_t^0(x,p,u)-\nabla_u \mathcal H_t^0(\tilde x,\tilde p,\tilde u)
\bigr| \\
&\le
|p|\,
\bigl\|\nabla_u b_t(x,u)-\nabla_u b_t(\tilde x,\tilde u)\bigr\|
+
\|\nabla_u b_t(\tilde x,\tilde u)\| \,|p-\tilde p|
+
\bigl\|\nabla_u f_t(x,u)-\nabla_u f_t(\tilde x,\tilde u)\bigr\|.
\end{aligned}
\]
By the mean value theorem,
\[
\bigl\|\nabla_u b_t(x,u)-\nabla_u b_t(\tilde x,\tilde u)\bigr\|
\le
M|x-\tilde x|+M_{b,uu}|u-\tilde u|,
\]
and similarly,
\[
\bigl\|\nabla_u f_t(x,u)-\nabla_u f_t(\tilde x,\tilde u)\bigr\|
\le
M|x-\tilde x|+M_{f,uu}|u-\tilde u|.
\]
Therefore,
\[
\begin{aligned}
&\bigl|
\nabla_u \mathcal H_t^0(x,p,u)-\nabla_u \mathcal H_t^0(\tilde x,\tilde p,\tilde u)
\bigr| \\
&\le
M_P\bigl(M|x-\tilde x|+M_{b,uu}|u-\tilde u|\bigr)
+
M|p-\tilde p|
+
M|x-\tilde x|+M_{f,uu}|u-\tilde u| \\
&=
M \bigl(M_P +1\bigr)|x-\tilde x|
+
M|p-\tilde p|
+
\bigl(M_P M_{b,uu}+M_{f,uu}\bigr)|u-\tilde u|.
\end{aligned}
\]
In summary, for all $t\in[0,T]$ and all
$(x,p,u),(\tilde x,\tilde p,\tilde u)$ with $|x|,|\tilde x|\le M_X$ and $|p|,|\tilde p|\le M_P$,
\begin{equation}\label{eq:H-grad-Lip}
\bigl|
\nabla_u \mathcal H_t^0(x,p,u)-\nabla_u \mathcal H_t^0(\tilde x,\tilde p,\tilde u)
\bigr|
\le
C_H\bigl(|x-\tilde x|+|p-\tilde p|+|u-\tilde u|\bigr),
\end{equation}
where
\[
C_H
:=
\max\Bigl\{
M(M_P+1),
\,
M_P M_{b,uu}+M_{f,uu}
\Bigr\}.
\]

Applying \eqref{eq:H-grad-Lip} with $(x,p,u)=(x_t^\theta,p_t^\theta,u_t^\theta)$ and
$(\tilde x,\tilde p,\tilde u)=(x_t^u,p_t^u,u_t)$, we obtain
\[
\begin{aligned}
|I_1|
&\le
C_H\int_0^1\int_0^T
\Bigl(
|x_t^\theta-x_t^u|
+
|p_t^\theta-p_t^u|
+
|u_t^\theta-u_t|
\Bigr)|\delta u_t|\dd t\dd \theta.
\end{aligned}
\]
By the Cauchy--Schwarz inequality,
\begin{align*}
    & \int_0^T
\Bigl(
|x_t^\theta-x_t^u|
+
|p_t^\theta-p_t^u|
+
|u_t^\theta-u_t|
\Bigr)|\delta u_t|\dd t \\
& \qquad \leq \Bigl( \| x^\theta - x^u\|_{L^2} + \| p^\theta - p^u \|_{L^2} + \theta \|\delta u\|_{L^2} \Bigr) \| \delta u\|_{L^2}.
\end{align*}
By Lemma~\ref{lem:state-adjoint-stability},
\[
\|x^\theta-x^u\|_{L^2}
+
\|p^\theta-p^u\|_{L^2}
\le
(C_x + C_p) \|u^\theta-u\|_{L^2}
=
\theta (C_x+ C_p) \|\delta u\|_{L^2}.
\]
Hence
\[
\begin{aligned}
|I_1|
&\le
C_H\int_0^1
\theta (C_x + C_p + 1) \| \delta u\|_{L^2}^2 \dd\theta
=
\dfrac{C_H}{2} (C_x + C_p+ 1) \| \delta u\|_{L^2}^2.
\end{aligned}
\]
Therefore, by \eqref{eq:H-unif-conv},
\[
|I_1|
\le
\frac{C_H(C_x + C_p +1)}{\sigma_h} \mathcal{D}_h(v\mid u).
\]

\smallskip

\noindent
\textit{Step 4: Conclusion.}
Combining the estimates for $I_1$ and $I_2$ with \eqref{eq:decomp-J}, we obtain
\[
J^\tau(v)-J^\tau(u)-dJ^\tau(u)(v-u)
\le
\left(
\tau+\frac{C_H(C_x + C_p +1)}{\sigma_h}
\right)\mathcal{D}_h(v\mid u).
\]
Thus \eqref{eq:relative-smoothness} holds with
\[
L
:=
\tau+\frac{C_H(C_x + C_p +1)}{\sigma_h}.
\]
This completes the proof.
\end{proof}

\section{Energy Dissipation}\label{sec:energy}

Using Lemma~\ref{lem:relative-smoothness}, we now show that the cost functional decreases monotonically along the iterates $\{u^n\}_{n \geq 0}$.

\begin{theorem}[Energy dissipation]\label{thm:energy}
Under Assumptions~\ref{asp:main_asp} and~\ref{ass:h-unif-conv}, let $\lambda\ge L$ and let $\{u^n\}_{n\ge 0}$ be the sequence generated by Algorithm~\ref{alg:mirror-pmp}. Then
\begin{equation}\label{eq:bregman-diss}
J^\tau(u^n)-J^\tau(u^{n+1})
\ge
(\lambda-L)\,\mathcal{D}_h(u^{n+1}\mid u^n) \geq 0, \qquad n\ge 0.
\end{equation}
In particular, if $\lambda > L$ and $\inf_{u\in\mathcal U}J^\tau(u)>-\infty$, then
\begin{equation}\label{eq:bregman-to-zero}
\mathcal{D}_h(u^{n+1}\mid u^n)\to 0.
\end{equation}
\end{theorem}

\begin{proof}
Apply \eqref{eq:relative-smoothness-H} with $u=u^n$ and $v=u^{n+1}$: 
\[
J^\tau(u^{n+1})-J^\tau(u^n)
\le
-\int_0^T \nabla_u \mathcal H_t^\tau(x_t^n,p_t^n,u_t^n)\cdot (u_t^{n+1}-u_t^n)\dd t
+
L \mathcal{D}_h(u^{n+1}\mid u^n).
\]
By the definition of the mirror step \eqref{eq:mirror-update}, for every $t\in[0,T]$,
\begin{multline*}
\nabla_u \mathcal H_t^\tau(x_t^n,p_t^n,u_t^n)\cdot (u_t^{n+1}-u_t^n)
- \lambda D_h(u_t^{n+1}\mid u_t^n) \ge \\
\ge
\nabla_u \mathcal H_t^\tau(x_t^n,p_t^n,u_t^n)\cdot (u_t^n-u_t^n)
- \lambda D_h(u_t^n\mid u_t^n)=0.
\end{multline*}
Integrating over $t$ yields
\[
\int_0^T \nabla_u \mathcal H_t^\tau(x_t^n,p_t^n,u_t^n)\cdot (u_t^{n+1}-u_t^n)\dd t
\ge
\lambda \mathcal{D}_h(u^{n+1}\mid u^n).
\]
Substituting this estimate back into the relative smoothness inequality,
\[
J^\tau(u^{n+1})-J^\tau(u^n)
\le
-(\lambda-L)\mathcal{D}_h(u^{n+1}\mid u^n)\le 0,
\]
since $\lambda\ge L$. This proves \eqref{eq:bregman-diss}. 

If, in addition, $\lambda > L$ and $\inf_{u\in\mathcal U}J^\tau(u)>-\infty$, then summing in $n$ gives
\[
\sum_{n=0}^{N-1} (\lambda-L) \mathcal{D}_h(u^{n+1}\mid u^n)
\le
J^\tau(u^0)-J^\tau(u^N)
\le
J^\tau(u^0)-\inf_{u\in\mathcal U}J^\tau(u).
\]
Since $\lambda - L >0$, we conclude that $\sum_{n\ge 0} \mathcal{D}_h(u^{n+1}\mid u^n)<\infty$ and, in particular,
\[
\mathcal{D}_h(u^{n+1}\mid u^n)\to 0. \qedhere
\]
\end{proof}

\section{Relative Convexity}\label{sec:re_convex}

In this section, we establish a relative convexity property of the cost functional with modulus $\tau$.

\begin{assump}[Convexity of the control problem]\label{ass:convexity}
The terminal cost $g$ is convex, and for each $t\in[0,T]$ and $p\in\mathbb R^d$ the map
\[
(x,u)\mapsto \mathcal H_t^0(x,p,u)
\]
is concave.
\end{assump}

\begin{remark}
While the uniform concavity condition seems restrictive, it is a standard assumption to guarantee the sufficiency of the Pontryagin's maximum principle, see \cite[Chapter 3, Theorem 2.5]{yong1999stochastic}. 
\end{remark}

\begin{theorem}[Relative convexity]\label{thm:relative-convexity}
Suppose that Assumptions~\ref{asp:main_asp} and~\ref{ass:convexity} hold. Then, for all $u,v\in\mathcal U$,
\begin{equation}\label{eq:relative-convexity}
J^\tau(v)-J^\tau(u)
\ge
dJ^\tau(u)(v-u)+\tau \mathcal{D}_h(v\mid u).
\end{equation}
Equivalently, by Lemma~\ref{lem:first-variation},
\begin{equation}\label{eq:relative-convexity-H}
J^\tau(v)-J^\tau(u)
\ge
- \int_0^T \nabla_u \mathcal H_t^\tau(x_t,p_t,u_t)\cdot (v_t-u_t)\dd t
+
\tau\int_0^T D_h(v_t\mid u_t)\dd t,
\end{equation}
where $(x,p)$ is the state-adjoint pair associated with $u$.
\end{theorem}

Thus the regularization parameter $\tau$ serves as a modulus of relative convexity with respect to the Bregman geometry induced by $h$.

\begin{proof}
Let $x$ and $x'$ be the states associated with $u$ and $v$, respectively, and let $p$ be the adjoint associated with $u$:
\[
\dot p_t=-\nabla_x \mathcal H_t^0(x_t,p_t,u_t),\qquad p_T= -\nabla g(x_T).
\]
We start from
\[
J^\tau(u)-J^\tau(v)
=
g(x_T)-g(x_T')
+\int_0^T \bigl(f_t(x_t,u_t)-f_t(x_t',v_t)\bigr)\dd t
+\tau \int_0^T \bigl(h(u_t)-h(v_t)\bigr)\dd t.
\]
By convexity of $g$,
\[
g(x_T)-g(x_T')\le \nabla g(x_T)\cdot(x_T-x_T')= -p_T\cdot(x_T-x_T').
\]
Using integration by parts and the state and adjoint equations,
\[
p_T\cdot(x_T-x_T')
=
\int_0^T p_t\cdot(\dot x_t-\dot x_t')\dd t
-
\int_0^T \nabla_x \mathcal H_t^0(x_t,p_t,u_t)\cdot(x_t-x_t')\dd t.
\]
Since $\dot x_t=b_t(x_t,u_t)$ and $\dot x_t'=b_t(x_t',v_t)$, we obtain
\[
p_T\cdot(x_T-x_T')
=
\int_0^T \Bigl[
p_t\cdot\bigl(b_t(x_t,u_t)-b_t(x_t',v_t)\bigr)
-
\nabla_x \mathcal H_t^0(x_t,p_t,u_t)\cdot(x_t-x_t')
\Bigr]\dd t.
\]
Therefore
\begin{align*}
J^\tau(u)-J^\tau(v)
&\le
- \int_0^T
\Bigl[
\mathcal H_t^0(x_t,p_t,u_t)-\mathcal H_t^0(x_t',p_t,v_t)
+
\nabla_x \mathcal H_t^0(x_t,p_t,u_t)\cdot(x_t'-x_t)
\Bigr]\dd t \\
&\quad
+\tau \int_0^T \bigl(h(u_t)-h(v_t)\bigr)\dd t.
\end{align*}
Now use the concavity of $(x,u)\mapsto \mathcal H_t^0(x,p_t,u)$:
\[
\mathcal H_t^0(x_t',p_t,v_t)
\le
\mathcal H_t^0(x_t,p_t,u_t)
+\nabla_x \mathcal H_t^0(x_t,p_t,u_t)\cdot(x_t'-x_t)
+\nabla_u \mathcal H_t^0(x_t,p_t,u_t)\cdot(v_t-u_t).
\]
Rearranging,
\[
\mathcal H_t^0(x_t,p_t,u_t)-\mathcal H_t^0(x_t',p_t,v_t)
+\nabla_x \mathcal H_t^0(x_t,p_t,u_t)\cdot(x_t'-x_t)
\ge
-\nabla_u \mathcal H_t^0(x_t,p_t,u_t)\cdot(v_t-u_t).
\]
Hence
\[
J^\tau(u)-J^\tau(v)
\le
\int_0^T \nabla_u \mathcal H_t^0(x_t,p_t,u_t)\cdot(v_t-u_t)\dd t
+\tau \int_0^T \bigl(h(u_t)-h(v_t)\bigr)\dd t.
\]
Using the identity
\[
h(v_t)-h(u_t)
=
D_h(v_t\mid u_t)+\nabla h(u_t)\cdot (v_t-u_t),
\]
that is,
\[
h(u_t)-h(v_t)
=
- D_h(v_t\mid u_t) - \nabla h(u_t)\cdot (v_t-u_t),
\]
we obtain
\[
J^\tau(u)-J^\tau(v)
\le
\int_0^T
\bigl(\nabla_u \mathcal H_t^0(x_t,p_t,u_t) - \tau \nabla h(u_t)\bigr)\cdot(v_t-u_t)\dd t
-
\tau \int_0^T D_h(v_t\mid u_t)\dd t,
\]
that is,
\[
J^\tau(v)-J^\tau(u)
\ge
-\int_0^T
\bigl(\nabla_u \mathcal H_t^0(x_t,p_t,u_t) - \tau \nabla h(u_t)\bigr)\cdot(v_t-u_t)\dd t
+
\tau \int_0^T D_h(v_t\mid u_t)\dd t.
\]
Since
\[
\nabla_u \mathcal H_t^\tau(x_t,p_t,u_t)
=
\nabla_u \mathcal H_t^0(x_t,p_t,u_t) - \tau \nabla h(u_t),
\]
the claimed inequality \eqref{eq:relative-convexity-H} follows.
\end{proof}

\section{Convergence Rates}\label{sec:conv_rates}

We now combine the relative smoothness and relative convexity properties to obtain convergence rates for the mirror descent algorithm.

\begin{lemma}[Three-point inequality]\label{lem:three-point}
Fix $u,\bar u,w\in U$ and $\xi\in\mathbb R^m$. Suppose that
\[
\bar u \in \arg\max_{v\in U}\bigl\{\xi\cdot (v-u) - \lambda D_h(v\mid u)\bigr\}.
\]
Then
\begin{equation}\label{eq:three-point}
\xi\cdot (w-u)- \lambda D_h(w\mid u)
\le
\xi\cdot (\bar u-u) - \lambda D_h(\bar u\mid u) - \lambda D_h(w\mid \bar u).
\end{equation}
\end{lemma}

\begin{proof}
Define
\[
\Phi(v):=\xi\cdot (v-u) - \lambda D_h(v\mid u)
=
\xi\cdot(v-u) - \lambda h(v) + \lambda h(u) + \lambda \nabla h(u)\cdot(v-u).
\]
Since $h$ is convex, $\Phi$ is concave, with
\[
\nabla \Phi(v)=\xi - \lambda \nabla h(v) + \lambda \nabla h(u).
\]
By optimality of $\bar u$ and the first-order variational inequality,
\[
\bigl(\xi - \lambda \nabla h(\bar u) + \lambda \nabla h(u)\bigr)\cdot (w-\bar u)\le 0
\qquad \forall w\in U.
\]
Rearranging,
\[
\xi\cdot (w-\bar u)
\le
\lambda\bigl(\nabla h(\bar u)-\nabla h( u)\bigr)\cdot (w-\bar u).
\]
Add $\xi\cdot(\bar u-u)$ to both sides to obtain
\[
\xi\cdot (w-u)
\le
\xi\cdot(\bar u-u) +
\lambda\bigl(\nabla h(\bar u)-\nabla h(u)\bigr)\cdot (w-\bar u).
\]
Now apply the standard Bregman identity
\[
D_h(w\mid u)-D_h(\bar u\mid u)-D_h(w\mid \bar u)
=
\bigl(\nabla h(\bar u)-\nabla h(u)\bigr)\cdot (w-\bar u),
\]
which yields \eqref{eq:three-point}.
\end{proof}

\begin{theorem}[Convergence rates]\label{thm:rates}
Under Assumptions~\ref{asp:main_asp},~\ref{ass:h-unif-conv}, and~\ref{ass:convexity}, let $\lambda\ge L$, where $L$ is the relative-smoothness constant in Lemma~\ref{lem:relative-smoothness}. Let $\{u^n\}_{n\ge 0}$ be the sequence generated by Algorithm~\ref{alg:mirror-pmp}.

\begin{itemize}
\item[(i)] If $\tau=0$, then for any $v\in\mathcal U$ such that $\mathcal{D}_h(v\mid u^n)<\infty$ for all $n$,
\[
J^0(u^n)-J^0(v)\le \frac{\lambda}{n}\mathcal{D}_h(v\mid u^0),\qquad n\ge 1.
\]

\item[(ii)] If $\tau>0$ and $u^\ast$ minimizes $J^\tau$ over $\mathcal U$, then
\[
0\le J^\tau(u^n)-J^\tau(u^\ast)
\le
\lambda\Bigl(1-\frac{\tau}{\lambda}\Bigr)^{n-1} \mathcal{D}_h(u^\ast\mid u^0),
\qquad n\ge 1.
\]
\end{itemize}
\end{theorem}

Thus, in the unregularized case ($\tau = 0$), the scheme attains the standard mirror-descent rate $O(1/n)$, while a positive regularization parameter ($\tau > 0$) improves the convergence to a geometric (exponential) rate.

\begin{proof}
Fix $n\ge 0$. By the mirror-step definition \eqref{eq:mirror-update}, applying Lemma~\ref{lem:three-point} pointwise in time with
\[
u=u_t^n,\qquad
\bar u=u_t^{n+1},\qquad
w=v_t,\qquad
\xi=\nabla_u \mathcal H_t^\tau(x_t^n,p_t^n,u_t^n),
\]
yields
\begin{multline*}
    \nabla_u \mathcal H_t^\tau(x_t^n,p_t^n,u_t^n)\cdot(v_t-u_t^n) - \lambda D_h(v_t\mid u_t^n) \\
    \le
\nabla_u \mathcal H_t^\tau(x_t^n,p_t^n,u_t^n)\cdot(u_t^{n+1}-u_t^n)
- \lambda D_h(u_t^{n+1}\mid u_t^n) - \lambda D_h(v_t\mid u_t^{n+1}).
\end{multline*}
Integrating in time gives
\begin{align}
&\int_0^T \nabla_u \mathcal H_t^\tau(x_t^n,p_t^n,u_t^n)\cdot(v_t-u_t^n)\dd t \notag\\
& \le
\int_0^T \nabla_u \mathcal H_t^\tau(x_t^n,p_t^n,u_t^n)\cdot(u_t^{n+1}-u_t^n)\dd t
-\lambda \mathcal{D}_h(u^{n+1}\mid u^n)
-\lambda \mathcal{D}_h(v\mid u^{n+1})
+\lambda \mathcal{D}_h(v\mid u^n).
\label{eq:key-3point-integrated}
\end{align}

On the other hand, the relative smoothness \eqref{eq:relative-smoothness-H} in Lemma~\ref{lem:relative-smoothness} applied to $(u^n,u^{n+1})$ gives
\[
J^\tau(u^{n+1})
\le
J^\tau(u^n)
-
\int_0^T \nabla_u \mathcal H_t^\tau(x_t^n,p_t^n,u_t^n)\cdot(u_t^{n+1}-u_t^n)\dd t
+
L \mathcal{D}_h(u^{n+1}\mid u^n).
\]
Substituting \eqref{eq:key-3point-integrated} and using $\lambda\ge L$ gives
\begin{align}
J^\tau(u^{n+1})
&\le
J^\tau(u^n)
-
\int_0^T \nabla_u \mathcal H_t^\tau(x_t^n,p_t^n,u_t^n)\cdot(v_t-u_t^n)\dd t \notag\\
&\qquad
+\lambda \mathcal{D}_h(v\mid u^n)-\lambda \mathcal{D}_h(v\mid u^{n+1})
-(\lambda-L)\mathcal{D}_h(u^{n+1}\mid u^n) \notag\\
&\le
J^\tau(u^n)
-
\int_0^T \nabla_u \mathcal H_t^\tau(x_t^n,p_t^n,u_t^n)\cdot(v_t-u_t^n)\dd t
+\lambda \mathcal{D}_h(v\mid u^n)-\lambda \mathcal{D}_h(v\mid u^{n+1}).
\label{eq:pre-convexity}
\end{align}

Now apply relative convexity \eqref{eq:relative-convexity-H} with $u=u^n$ and $v$ arbitrary:
\[
J^\tau(v)
\ge
J^\tau(u^n)
-
\int_0^T \nabla_u \mathcal H_t^\tau(x_t^n,p_t^n,u_t^n)\cdot(v_t-u_t^n)\dd t
+
\tau \mathcal{D}_h(v\mid u^n).
\]
Substituting this lower bound for the middle term in \eqref{eq:pre-convexity}, we obtain
\begin{equation}\label{eq:master-ineq}
J^\tau(u^{n+1})-J^\tau(v)
\le
(\lambda-\tau)\mathcal{D}_h(v\mid u^n)-\lambda \mathcal{D}_h(v\mid u^{n+1}).
\end{equation}

\paragraph{Case 1: \texorpdfstring{$\tau=0$}{tau=0}.}
Then \eqref{eq:master-ineq} becomes
\[
J^0(u^{n+1})-J^0(v)
\le
\lambda \mathcal{D}_h(v\mid u^n)-\lambda \mathcal{D}_h(v\mid u^{n+1}).
\]
Summing from $n=0$ to $m-1$,
\[
\sum_{n=0}^{m-1}\bigl(J^0(u^{n+1})-J^0(v)\bigr)
\le
\lambda \mathcal{D}_h(v\mid u^0)-\lambda \mathcal{D}_h(v\mid u^m)
\le
\lambda \mathcal{D}_h(v\mid u^0).
\]
By the energy dissipation property, $J^0(u^m)\le J^0(u^{n+1})$ for all $n=0,\dots,m-1$, so
\[
m\bigl(J^0(u^m)-J^0(v)\bigr)
\le
\sum_{n=0}^{m-1}\bigl(J^0(u^{n+1})-J^0(v)\bigr)
\le
\lambda \mathcal{D}_h(v\mid u^0).
\]
Hence
\[
J^0(u^m)-J^0(v)\le \frac{\lambda}{m}\mathcal{D}_h(v\mid u^0).
\]

\paragraph{Case 2: \texorpdfstring{$\tau>0$}{tau>0}.}
Let $v=u^\ast$, where $u^\ast$ minimizes $J^\tau$. Since
\[
J^\tau(u^{n+1})-J^\tau(u^\ast)\ge 0,
\]
the inequality \eqref{eq:master-ineq} implies
\[
\lambda \mathcal{D}_h(u^\ast\mid u^{n+1})
\le
(\lambda-\tau)\mathcal{D}_h(u^\ast\mid u^n),
\]
that is,
\[
\mathcal{D}_h(u^\ast\mid u^{n+1})
\le
\Bigl(1-\frac{\tau}{\lambda}\Bigr)\mathcal{D}_h(u^\ast\mid u^n).
\]
Iterating,
\[
\mathcal{D}_h(u^\ast\mid u^n)
\le
\Bigl(1-\frac{\tau}{\lambda}\Bigr)^n \mathcal{D}_h(u^\ast\mid u^0).
\]
Finally, returning to \eqref{eq:master-ineq} with $v=u^\ast$ and dropping the nonpositive term
$-\lambda \mathcal{D}_h(u^\ast\mid u^{n+1})$, we obtain
\[
J^\tau(u^{n+1})-J^\tau(u^\ast)
\le
(\lambda-\tau)\mathcal{D}_h(u^\ast\mid u^n)
\le
\lambda \mathcal{D}_h(u^\ast\mid u^n).
\]
Hence
\[
J^\tau(u^{n+1})-J^\tau(u^\ast)
\le
\lambda \Bigl(1-\frac{\tau}{\lambda}\Bigr)^n \mathcal{D}_h(u^\ast\mid u^0).
\]
Reindexing $n\mapsto n-1$ yields the claimed geometric estimate.
\end{proof}

\section{Numerical Examples}\label{sec:numerics}

In this section, we present three numerical examples illustrating the behavior of the mirror-descent iteration in different settings. We begin with a one-dimensional linear-quadratic problem for which the optimal solution is available explicitly, providing a simple benchmark for the algorithm. We then consider a degenerate convex example to illustrate the qualitative difference between the unregularized and regularized cases. Finally, we examine a nonlinear system with coupled dynamics in increasing state dimensions to explore the behavior of the method beyond the simple one-dimensional settings. These experiments are intended primarily as numerical illustrations of the convergence phenomena discussed in the preceding sections.

\subsection{A one-dimensional linear-quadratic example}

We consider a one-dimensional linear-quadratic (LQ) optimal control problem for which the optimal solution is explicit. Let
\[
\dot x_t = a x_t + u_t, \qquad x_0 = \xi, \qquad t \in [0,T],
\]
and define the regularized cost functional
\[
J^\tau(u)
=
\frac{1}{2}\int_0^T \Big( q x_t^2 + \tau u_t^2 \Big)\dd t
+
\frac{s}{2} x_T^2,
\]
where $a\in\mathbb{R}$, $q\ge 0$, $s\ge 0$. Here $\tau$ is the effective quadratic weight on the control, which also plays the role of regularization. To ensure the existence of an optimal solution, we set $\tau > 0$.

This corresponds to the data
\[
b(x,u)=ax+u,\qquad
f(x,u)=\frac{1}{2}q x^2,\qquad
h(u)=\frac{1}{2}u^2,\qquad
g(x)=\frac{s}{2}x^2.
\]

The regularized Hamiltonian is given by
\[
\mathcal H_t^\tau(x,p,u)
=
p(ax+u) - \frac{1}{2}q x^2 - \frac{1}{2} \tau u^2.
\]

The Pontryagin maximum principle yields the state and adjoint equations
\[
\dot x_t = a x_t + u_t,\qquad x_0=\xi,
\]
\[
\dot p_t = -\nabla_x \mathcal H_t^0(x_t,p_t,u_t) = -a p_t + q x_t,\qquad
p_T = -s x_T,
\]
together with the maximum condition
\[
0 = \nabla_u \mathcal H_t^\tau(x_t,p_t,u_t) = p_t - \tau u_t.
\]
Hence the optimal control satisfies
\[
u_t^* = \frac{p_t^*}{\tau}.
\]

Introducing the ansatz $p_t^* = -P(t)x_t^*$, one obtains the Riccati equation
\[
P'(t)
=
\frac{1}{\tau}P(t)^2 - 2aP(t) - q,
\qquad
P(T)=s.
\]
In one dimension, this equation admits an explicit solution. Setting
\[
\gamma := \sqrt{a^2 + \frac{q}{\tau}},
\]
and defining
\[
P_\pm = \tau(a \pm \gamma),
\qquad
\kappa = \frac{s - P_+}{s - P_-},
\]
we obtain
\[
P(t)
=
\frac{P_+ - \kappa e^{-2\gamma(T-t)} P_-}
{1 - \kappa e^{-2\gamma(T-t)}}.
\]
The optimal state and control are then given by
\[
x_t^*
=
\xi \exp\!\left(\int_0^t \Big(a - \frac{P(\sigma)}{\tau}\Big)\dd\sigma\right),
\qquad
u_t^*
=
-\frac{P(t)}{\tau} x_t^*.
\]

In the Euclidean case $h(u)=\frac12|u|^2$ and $U=\mathbb{R}$, the mirror descent update reduces to
\[
u_t^{n+1}
=
\left(1-\frac{\tau}{\lambda}\right)u_t^n
+
\frac{1}{\lambda}\nabla_u \mathcal H_t^0(x_t^n,p_t^n,u_t^n).
\]
Since $\nabla_u \mathcal H_t^0(x,p,u)=p$,
we obtain the explicit update
\[
u_t^{n+1}
=
\left(1-\frac{\tau}{\lambda}\right)u_t^n
+
\frac{1}{\lambda}p_t^n.
\]

In the numerical experiments, we take
\[
T=1,\quad \xi=0.5,\quad a=1,\quad q=1,\quad s=1,\quad \lambda=30,
\]
and initialize the algorithm with a constant control $u_t^0 \equiv 4$. The result is plotted in Figure~\ref{fig:gen_lqr-tau-comparison}. We observe clearly geometric convergence. 

\begin{figure}[ht]
    \centering
    \includegraphics[width=0.70\textwidth]{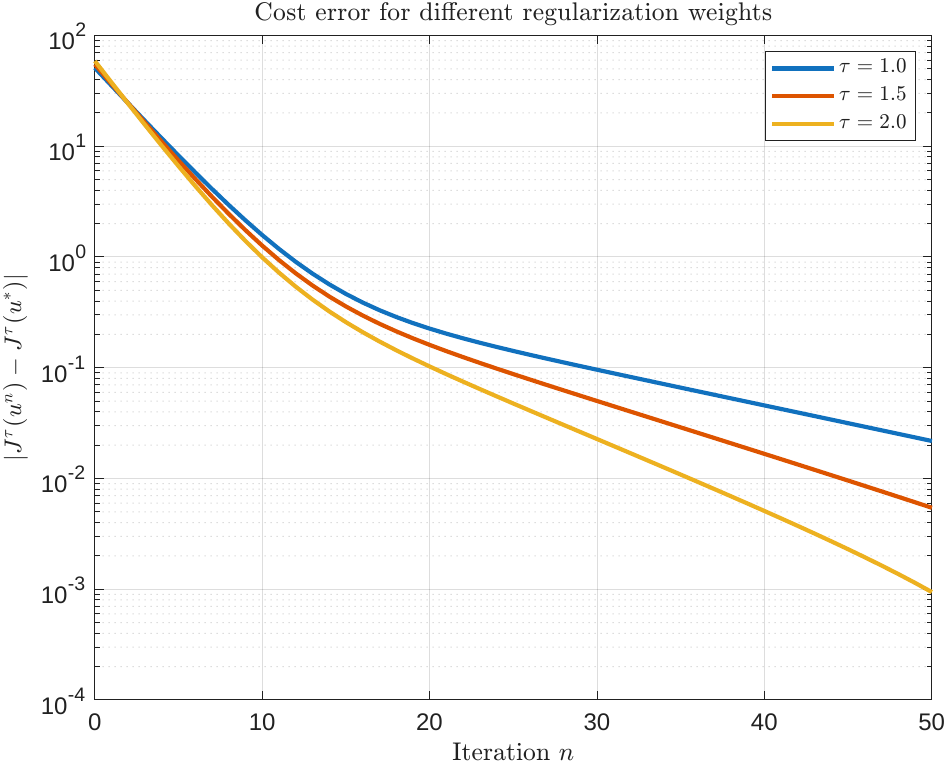}
    \caption{Objective convergence of Algorithm~\ref{alg:mirror-pmp} for the one-dimensional linear-quadratic example with different values of the regularization parameter $\tau$. The vertical axis shows the cost error $|J^{\tau}(u^n)-J^{\tau}(u^*)|$ on a logarithmic scale.}
    \label{fig:gen_lqr-tau-comparison}
\end{figure}

\subsection{A one-dimensional quartic terminal example illustrating the rate difference between \texorpdfstring{$\tau=0$}{tau=0} and \texorpdfstring{$\tau>0$}{tau>0}}

We next present a simple one-dimensional control problem in which the convergence behavior differs qualitatively between the cases $\tau=0$ and $\tau>0$, due to the lack of strong convexity.

Consider state dynamics
\[
\dot x_t = u_t, \qquad x_0 = 0, \qquad t\in[0,T],
\]
and define the cost functional
\[
J^\tau(u)
=
\frac14 x_T^4
+
\frac{\tau}{2}\int_0^T u_t^2\dd t,
\qquad x_T=\int_0^T u_t\dd t.
\]
For $\tau > 0$, the unique optimal control is
\[
u_t^* \equiv 0,
\qquad
x_t^* \equiv 0.
\]
For $\tau = 0$, every control with zero terminal state is optimal, including $u_t^* \equiv 0$.

The regularized Hamiltonian is
\[
\mathcal H_t^\tau(x,p,u)=pu-\frac{\tau}{2}u^2.
\]
The adjoint equation is
\[
\dot p_t = 0,
\qquad
p_T = -x_T^3.
\]
Hence
\[
p_t \equiv p_T = -x_T^3.
\]

In the Euclidean case $h(u)=\frac12 u^2$ and $U=\mathbb{R}$, Algorithm~\ref{alg:mirror-pmp} reduces to
\[
u_t^{n+1}
=
u_t^n + \frac1\lambda\bigl(p_t^n-\tau u_t^n\bigr).
\]
If the initial guess is chosen to be constant, say
\[
u_t^0\equiv \alpha_0,
\]
then all iterates remain constant in time, and we may write
\[
u_t^n\equiv \alpha_n.
\]
Since
\[
x_t^n=\alpha_n t,
\qquad
x_T^n=T\alpha_n,
\qquad
p_t^n\equiv -(T\alpha_n)^3,
\]
the algorithm reduces to the scalar iteration
\[
\alpha_{n+1}
=
\alpha_n - \frac1\lambda\bigl(T^3\alpha_n^3+\tau\alpha_n\bigr).
\]
Thus, when $\tau = 0$, the linear term vanishes and the iteration reduces to
\[
\alpha_{n+1} = \alpha_n - \frac{T^3}{\lambda}\alpha_n^3,
\]
so that the leading-order behavior is governed by the cubic term, yielding only sublinear convergence.

In contrast, when $\tau > 0$, the additional linear damping term $-\frac{\tau}{\lambda}\alpha_n$ dominates near the optimum $\alpha^* = 0$. More precisely, as $\alpha_n \to 0$, the cubic term is of higher order and can be neglected, so that the iteration can be linearized as
\[
\alpha_{n+1} \approx \left(1 - \frac{\tau}{\lambda}\right)\alpha_n.
\]
Consequently, the iteration behaves asymptotically like a linear contraction, and the convergence becomes geometric.

In the numerical experiments, we set $T = 1$ and $\lambda = 10$, and initialize the algorithm with a constant control $u_t^0 \equiv 2$. The results are shown in Figure~\ref{fig:quartic-tau-comparison}, which illustrates the different convergence behaviors for $\tau = 0$ and $\tau > 0$.

\begin{figure}[ht]
    \centering

    \begin{subfigure}{0.48\textwidth}
        \centering
        \includegraphics[width=\textwidth]{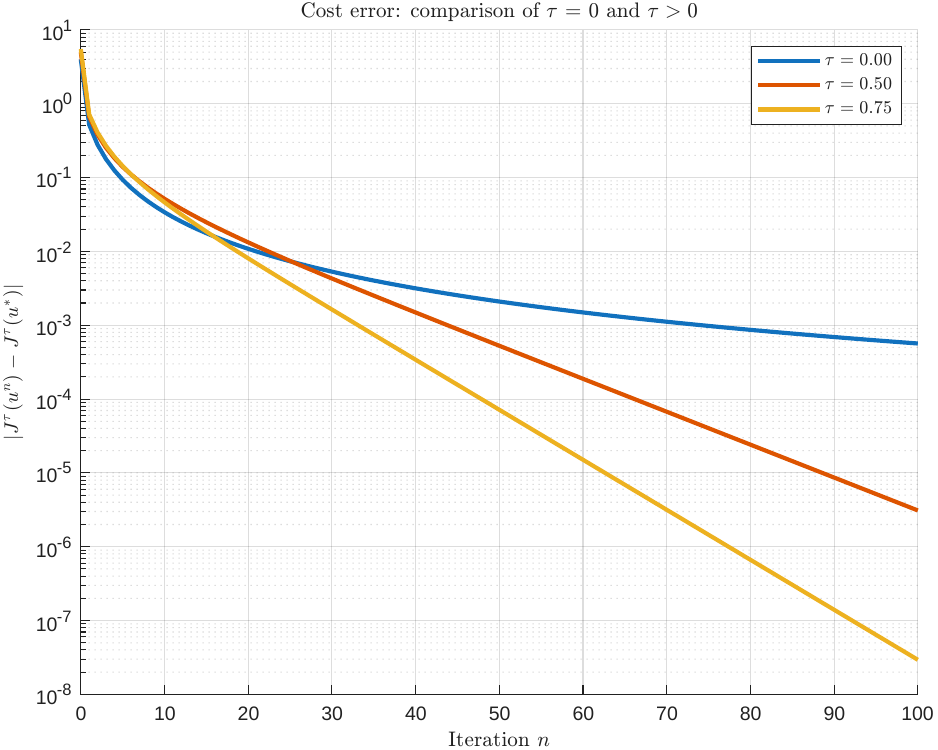}
        \caption{Semilog scale}
    \end{subfigure}
    \hfill
    \begin{subfigure}{0.48\textwidth}
        \centering
        \includegraphics[width=\textwidth]{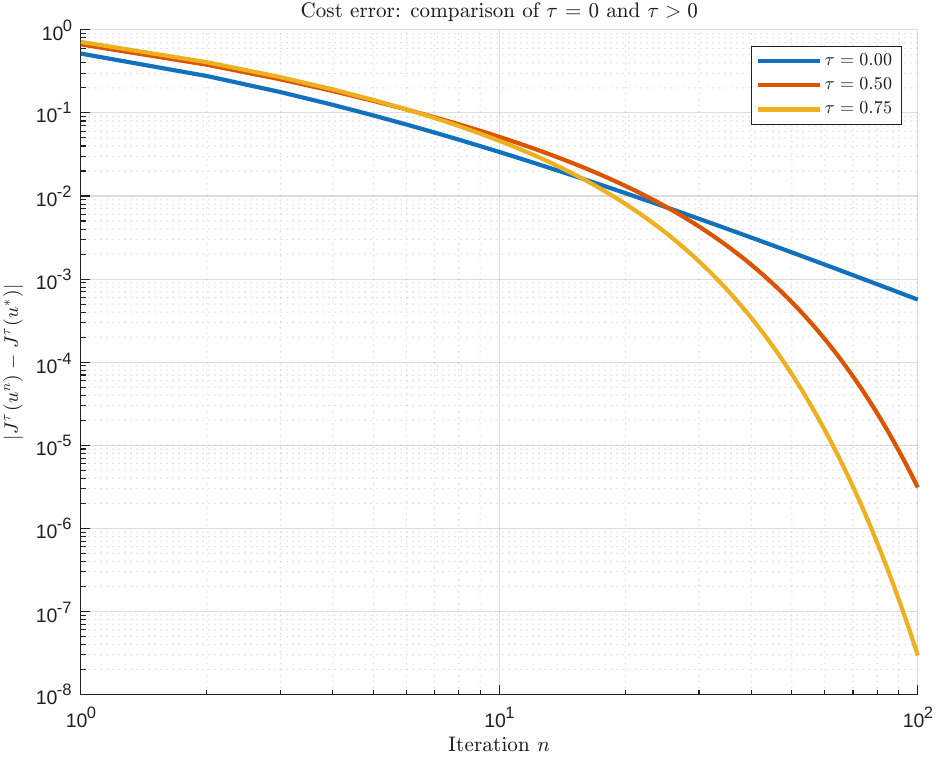}
        \caption{Log-log scale}
    \end{subfigure}

    \caption{Objective convergence of Algorithm~\ref{alg:mirror-pmp} for the one-dimensional quartic terminal example with different values of the regularization parameter $\tau$. The cost error $|J^\tau(u^n) - J^\tau(u^*)|$ is shown on semilog (left) and log-log (right) scales, highlighting the difference between exponential convergence for $\tau>0$ and sublinear behavior when $\tau=0$.}

    \label{fig:quartic-tau-comparison}
\end{figure}

Therefore, this example provides a simple control problem in which the regularization parameter $\tau$ qualitatively changes the convergence rate: the unregularized case is only sublinear, whereas any positive regularization yields geometric convergence near the optimum.

\subsection{A high-dimensional nonlinear example with coupled dynamics}

We next study how the mirror descent behaves as the state dimension increases. To this end, we consider a family of fully coupled nonlinear
control problems in dimensions
\[
d\in\{5,10,20\}.
\]
For each dimension $d$, we consider the controlled system
\[
\dot x_t = A x_t + B u_t + \gamma \sin(Cx_t),
\qquad
x_0=x^{\rm init},
\qquad
t\in[0,T],
\]
where $x_t\in\mathbb R^d$ and the control $u_t\in\mathbb R^d$. The sine nonlinearity is applied componentwise, and the matrices
\[
A,B,C\in\mathbb R^{d\times d}
\]
are chosen to be dense so that the coordinates are genuinely coupled and the
problem is not reducible to independent one-dimensional subsystems.

The cost functional is given by
\[
J^\tau(u)
=
\frac{q}{2d} \int_0^T
|x_t|^2 \,\dd t
+
\frac{s}{2d}|x_T-x^{\rm tar}|^2
+
\frac{\tau}{2} \int_0^T |u_t|^2\dd t.
\]
The normalization by $d$ is included so that the objective remains comparable across dimensions.

The corresponding Hamiltonian is
\[
\mathcal H_t^0(x,p,u)
=
p^\top\bigl(Ax+Bu+\gamma\sin(C x)\bigr)
-
\frac{q}{2d}|x|^2.
\]
The adjoint equation takes the form
\[
\dot p_t
=
-
\left(
A^\top
+
\gamma C^\top \operatorname{diag}(\cos(Cx_t))
\right)p_t
+
\frac{q}{d}x_t,
\qquad
p_T
=
-\frac{s}{d}(x_T-x^{\rm tar}).
\]

For $h(u)=\frac12\|u\|^2$ and $U=\mathbb R^d$, the Euclidean mirror update becomes
\[
u_t^{n+1}
=
u_t^n
+
\frac1\lambda
\left(
B^\top p_t^n
-
\tau u_t^n
\right).
\]

In the numerical experiments, we fix $T=1$ and discretize the time interval using $N_t=500$ uniform steps. The state dimension is varied over $d\in\{5,10,20\}$. The system matrices are generated randomly as
\[
A = 0.15 \frac{M_1}{\sqrt{d}} - 0.6 I,
\quad
B = 0.6 \frac{M_2}{\sqrt{d}} + 0.3 I,
\quad
C = 0.8 \frac{M_3}{\sqrt{d}},
\]
where $M_1, M_2, M_3$ have i.i.d.\ standard normal entries. The cost parameters are chosen as
\[
q = 1, \quad s = 5, \quad \gamma =1, \quad \tau = 0.5.
\]
The mirror descent inverse step size is set to $\lambda=20$. The initial state and target are defined componentwise by
\[
x^{\mathrm{init}}_i = 0.4\sin\left(\frac{i}{d}\pi\right),
\quad
x^{\mathrm{tar}}_i = 0.8\cos\left(\frac{i}{d}\pi\right), \qquad 1 \leq i \leq d.
\]
The initial control is chosen as
\[
u_t^0 = 2\sin(2\pi t)\mathbf{1}
+ 0.5\cos(4\pi t)\,v,
\]
where $v$ is a fixed vector with entries uniformly spaced in $[-1,1]$.

The convergence behavior is illustrated in Figure~\ref{fig:high-dim-comparison}, where we compare the objective error across different state dimensions $d \in \{5,10,20\}$. For each experiment, the algorithm is run up to $N=1000$ iterations, while only the initial portion $n=0,\dots,200$ is displayed in the figure to highlight the pre-asymptotic regime.

We observe that all dimensions exhibit approximately linear decay on the semilog scale over the displayed range. Moreover, while the objective decreases at a similar rate, higher-dimensional problems require more iterations to achieve the same level of accuracy.

\begin{figure}[ht]
    \centering
    \includegraphics[width=0.70\textwidth]{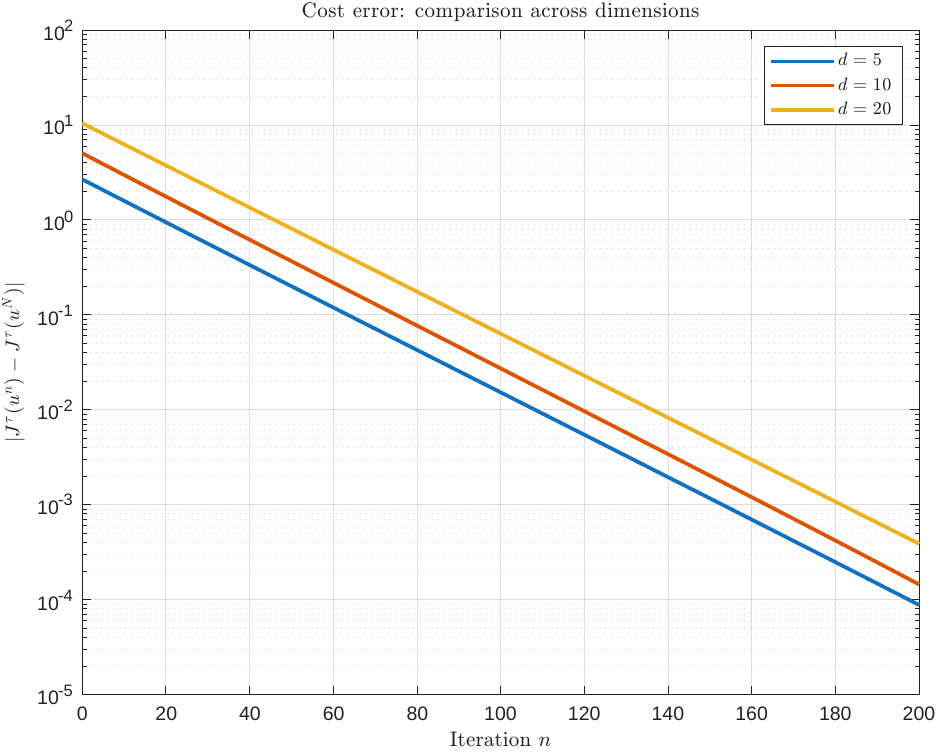}
    \caption{Objective convergence of Algorithm~\ref{alg:mirror-pmp} for the high-dimensional nonlinear example with coupled dynamics, comparing different state dimensions $d$. The cost error $|J^\tau(u^n) - J^\tau(u^N)|$, a surrogate gap relative to the final iterate, is plotted on a semilog scale. All curves exhibit exponential decay, while higher-dimensional problems require more iterations to achieve the same level of accuracy.}
    \label{fig:high-dim-comparison}
\end{figure}

\appendix

\section{Proof of the First Variation Formula}\label{sec:prf}
\begin{proof}[Proof of Lemma~\ref{lem:first-variation}]
Fix $u,v\in\mathcal U$ and set
\[
\delta u := v-u.
\]
For $\varepsilon \in [0,1]$, define the perturbed control
\[
u^\varepsilon := u+\varepsilon \delta u.
\]
Since $U$ is convex, $u^\varepsilon_t\in U$ for every $t$ whenever $\varepsilon\in[0,1]$,
hence $u^\varepsilon\in\mathcal U$. We compute the one-sided directional derivative at $\varepsilon=0$ along this feasible segment.

Let $x$ and $x^\varepsilon$ denote the states associated with $u$ and $u^\varepsilon$ respectively:
\begin{align*}
& \dot x_t=b_t(x_t,u_t),
\qquad
x(0)=x_0; \\
& \dot x_t^\varepsilon=b_t(x_t^\varepsilon,u_t^\varepsilon),
\qquad
x^\varepsilon(0)=x_0.
\end{align*}

We first identify the first-order variation of the state.
Define $y$ as the solution of the linearized equation
\begin{equation}\label{eq:linearized-state}
\dot y_t
=
\nabla_x b_t(x_t,u_t)\,y_t+\nabla_u b_t(x_t,u_t)\,\delta u_t,
\qquad
y_0=0.
\end{equation}

We claim that
\begin{equation}\label{eq:state-diff-quotient}
\frac{x^\varepsilon-x}{\varepsilon}\to y
\qquad\text{in } C([0,T];\mathbb R^d)
\quad\text{as }\varepsilon\to 0^+.
\end{equation}

To prove this, define the remainder
\[
r_t^\varepsilon
:=
x_t^\varepsilon-x_t-\varepsilon y_t.
\]
Then $r_0^\varepsilon=0$, and
\[
\dot r_t^\varepsilon
=
b_t(x_t^\varepsilon,u_t^\varepsilon)-b_t(x_t,u_t)
-
\varepsilon\Big(
\nabla_x b_t(x_t,u_t)y_t+\nabla_u b_t(x_t,u_t)\delta u_t
\Big).
\]
Now write
\[
x_t^\varepsilon=x_t+\varepsilon y_t+r_t^\varepsilon,
\qquad
u_t^\varepsilon=u_t+\varepsilon\delta u_t.
\]
Using the first-order Taylor expansion of $b_t(\cdot,\cdot)$ in the variables $(x,u)$
around $(x_t,u_t)$, we obtain
\[
b_t(x_t^\varepsilon,u_t^\varepsilon)-b_t(x_t,u_t)
=
\nabla_x b_t(x_t,u_t)\bigl(\varepsilon y_t+r_t^\varepsilon\bigr)
+
\nabla_u b_t(x_t,u_t)\,\varepsilon\delta u_t
+
\rho_t^\varepsilon,
\]
where the remainder $\rho_t^\varepsilon$ satisfies
\[
|\rho_t^\varepsilon|
\le
C\Big(|\varepsilon y_t+r_t^\varepsilon|+|\varepsilon\delta u_t|\Big)^2
\]
for a constant $C>0$ independent of $\varepsilon$ and $t$.
Hence
\[
\dot r_t^\varepsilon
=
\nabla_x b_t(x_t,u_t) r_t^\varepsilon + \rho_t^\varepsilon.
\]
Therefore
\begin{equation}\label{eq:re_e}
    |r_t^\varepsilon|
\le
\int_0^t |\nabla_x b_s(x_s,u_s)|\,|r_s^\varepsilon|\dd s
+
\int_0^t |\rho_s^\varepsilon|\dd s.
\end{equation}
By Assumption~\ref{asp:main_asp}, $\sup_{s \in [0,T]} |\nabla_x b(x_s, u_s)|\leq M$.
Moreover,
\begin{equation}\label{eq:rho_e}
    |\rho_t^\varepsilon|
\le
C\Bigl(\varepsilon^2|y_t|^2+|r_t^\varepsilon|^2+\varepsilon^2|\delta u_t|^2\Bigr).
\end{equation}
We estimate the square term $|r_t^{\varepsilon}|^2$ as follows. By the definition of $r_t^{\varepsilon}$,
\[
\sup_{t\in[0,T]} |r_t^\varepsilon|
\le
\sup_{t\in[0,T]} |x_t^\varepsilon-x_t|
+
\varepsilon\,\sup_{t\in[0,T]} |y_t|.
\]
By Assumption~\ref{asp:main_asp} (2) and $x_0^{\varepsilon} = x_0$,
\[
|x_t^\varepsilon - x_t| \leq \int_0^t |b_s(x_s^\varepsilon, u_s^\varepsilon) - b_s(x_s, u_s)| \dd s \leq \int_0^t M \left( |x_s^\varepsilon - x_s| + |u_s^\varepsilon - u_s| \right) \dd s.
\]
By Gr\"onwall's inequality,
\[
\sup_{t\in[0,T]} |x_t^\varepsilon-x_t|
\le
M e^{MT} \|u^\varepsilon-u\|_{L^1}
=
M e^{MT} \varepsilon\,\|\delta u\|_{L^1},
\]
By the linearized equation \eqref{eq:linearized-state} and $y_0 = 0$,
\[
|y_t| \leq \int_0^t |\nabla_x b_s(x_s, u_s)|\,|y_s|\dd s + \int_0^t |\nabla_u b_s(x_s, u_s)| \, |\delta u_s| \dd s.
\]
Assumption~\ref{asp:main_asp} yields
\[
|y_t| \leq M \left( \int_0^t |y_s| \dd s + \int_0^t |\delta u_s| \dd s \right).
\]
By Gr\"onwall's inequality,
\[
\sup_{t \in [0,T]} |y_t| \leq M e^{MT}\|\delta u\|_{L^1}.
\]
Therefore,
\[
\sup_{t\in[0,T]} |r_t^\varepsilon|
\le
C' \varepsilon, \qquad C':= 2 Me^{MT} \|\delta u\|_{L^1}.
\]
Substituting back to \eqref{eq:rho_e} and
using that $y\in C([0,T];\mathbb R^d)$ and $\delta u\in L^2([0,T];\mathbb R^m)$, we obtain
\[
\int_0^T |\rho_t^\varepsilon| \dd t = o(\varepsilon).
\]
By applying Gr\"onwall to \eqref{eq:re_e},
\[
\sup_{t\in[0,T]} \left|\frac{r_t^\varepsilon}{\varepsilon}\right|\to 0,
\]
which proves \eqref{eq:state-diff-quotient}.

We next compute the derivative of the cost.
By definition,
\[
J^0(u^\varepsilon)-J^0(u)
=
\int_0^T \Big(f_t(x_t^\varepsilon,u_t^\varepsilon)-f_t(x_t,u_t)\Big)\dd t
+
g(x_T^\varepsilon)-g(x_T).
\]
We treat the running cost and terminal cost separately.

For the running cost, by the first-order Taylor expansion of $f_t(\cdot,\cdot)$
around $(x_t,u_t)$,
\[
f_t(x_t^\varepsilon,u_t^\varepsilon)-f_t(x_t,u_t)
=
\nabla_x f_t(x_t,u_t)\cdot (x_t^\varepsilon-x_t)
+
\nabla_u f_t(x_t,u_t)\cdot (u_t^\varepsilon-u_t)
+
\tilde\rho_t^\varepsilon,
\]
where
\[
|\tilde\rho_t^\varepsilon|
\le
C\Big(|x_t^\varepsilon-x_t|+|u_t^\varepsilon-u_t|\Big)^2.
\]
Since
\[
x_t^\varepsilon-x_t=\varepsilon y_t+r_t^\varepsilon,
\qquad
u_t^\varepsilon-u_t=\varepsilon \delta u_t,
\]
and
\[
\sup_{t\in[0,T]}|r_t^\varepsilon|=o(\varepsilon),
\]
it follows that
\[
\int_0^T |\tilde\rho_t^\varepsilon|\dd t=o(\varepsilon).
\]
Hence
\[
\int_0^T \Big(f_t(x_t^\varepsilon,u_t^\varepsilon)-f_t(x_t,u_t)\Big)\dd t
=
\varepsilon\int_0^T
\Big(
\nabla_x f_t(x_t,u_t)\cdot y_t
+
\nabla_u f_t(x_t,u_t)\cdot \delta u_t
\Big)\dd t
+o(\varepsilon).
\]

For the terminal cost, since $g$ is continuously differentiable,
\[
g(x_T^\varepsilon)-g(x_T)
=
\nabla g(x_T)\cdot (x_T^\varepsilon-x_T)+o(|x_T^\varepsilon-x_T|).
\]
Using
\[
x_T^\varepsilon-x_T=\varepsilon y_T+r_T^\varepsilon
\qquad\text{and}\qquad
r_T^\varepsilon=o(\varepsilon),
\]
we obtain
\[
g(x_T^\varepsilon)-g(x_T)
=
\varepsilon\,\nabla g(x_T)\cdot y_T+o(\varepsilon).
\]

Combining the two expansions, we conclude that
\[
J^0(u^\varepsilon)-J^0(u)
=
\varepsilon\int_0^T
\Big(
\nabla_x f_t(x_t,u_t)\cdot y_t
+
\nabla_u f_t(x_t,u_t)\cdot \delta u_t
\Big)\dd t
+
\varepsilon\,\nabla g(x_T)\cdot y_T
+
o(\varepsilon).
\]
Therefore,
\begin{equation}\label{eq:pre-adjoint-first-variation}
dJ^0(u)(\delta u)
=
\int_0^T
\Big(
\nabla_x f_t(x_t,u_t)\cdot y_t
+
\nabla_u f_t(x_t,u_t)\cdot \delta u_t
\Big)\dd t
+
\nabla g(x_T)\cdot y_T.
\end{equation}

It remains to eliminate the state variation $y$ using the adjoint equation.
Recall that the adjoint $p$ satisfies
\[
\dot p_t
=
-\nabla_x \mathcal H_t^0(x_t,p_t,u_t),
\qquad
p_T= - \nabla g(x_T).
\]
Since
\[
\mathcal H_t^0(x,p,u)=p\cdot b_t(x,u) - f_t(x,u),
\]
we have
\[
\dot p_t
=
-\nabla_x b_t(x_t,u_t)^\top p_t + \nabla_x f_t(x_t,u_t).
\]
We now compute
\[
\frac{d}{dt}(p_t\cdot y_t)=\dot p_t\cdot y_t+p_t\cdot \dot y_t.
\]
Using the equations for $p$ and $y$, we get
\begin{align*}
\frac{d}{dt}(p_t\cdot y_t)
&=
\Big(-\nabla_x b_t(x_t,u_t)^\top p_t + \nabla_x f_t(x_t,u_t)\Big)\cdot y_t \\
&\quad
+
p_t\cdot
\Big(
\nabla_x b_t(x_t,u_t)y_t+\nabla_u b_t(x_t,u_t)\delta u_t
\Big).
\end{align*}
The two terms involving $\nabla_x b_t$ cancel, and thus
\[
\frac{d}{dt}(p_t\cdot y_t)
=
\nabla_x f_t(x_t,u_t)\cdot y_t
+
p_t\cdot \nabla_u b_t(x_t,u_t)\delta u_t.
\]
Integrating over $[0,T]$ and using $y_0=0$ and $p_T= - \nabla g(x_T)$, we obtain
\[
- \nabla g(x_T)\cdot y_T
=
\int_0^T \nabla_x f_t(x_t,u_t)\cdot y_t\dd t
+
\int_0^T p_t\cdot \nabla_u b_t(x_t,u_t)\delta u_t\dd t.
\]
Substituting this identity into \eqref{eq:pre-adjoint-first-variation}, the
terms involving $\nabla_x f_t(x_t,u_t)\cdot y_t$ cancel and we obtain
\[
dJ^0(u)(\delta u)
=
\int_0^T
\Big(
- p_t\cdot \nabla_u b_t(x_t,u_t)\delta u_t
+
\nabla_u f_t(x_t,u_t)\cdot \delta u_t
\Big)\dd t.
\]
Equivalently,
\[
dJ^0(u)(\delta u)
=
\int_0^T
\Big(
-\nabla_u b_t(x_t,u_t)^\top p_t+\nabla_u f_t(x_t,u_t)
\Big)\cdot \delta u_t\dd t.
\]
Since
\[
\nabla_u \mathcal H_t^0(x_t,p_t,u_t)
=
\nabla_u b_t(x_t,u_t)^\top p_t - \nabla_u f_t(x_t,u_t),
\]
we conclude that
\[
dJ^0(u)(\delta u)
=
- \int_0^T \nabla_u \mathcal H_t^0(x_t,p_t,u_t)\cdot \delta u_t\dd t.
\]
Recalling that $\delta u=v-u$, this proves \eqref{eq:first-var-J0}.

Finally, since
\[
J^\tau(u)=J^0(u)+\tau \int_0^T h(u_t)\dd t,
\]
and $h$ is continuously differentiable, we have
\[
dJ^\tau(u)(\delta u)
=
dJ^0(u)(\delta u)
+
\tau \int_0^T \nabla h(u_t)\cdot \delta u_t\dd t.
\]
Therefore
\[
dJ^\tau(u)(\delta u)
=
- \int_0^T \nabla_u \mathcal H_t^\tau(x_t,p_t,u_t)\cdot \delta u_t\dd t,
\]
which proves \eqref{eq:first-var-Jtau}.
\end{proof}

\section*{Acknowledgment}
This research is supported in part by the National Science Foundation via awards IIS-2403276 and DMS-2309378.

\bibliographystyle{plain}
\bibliography{refs}

@book{yong1999stochastic,
  title={Stochastic controls: Hamiltonian systems and HJB equations},
  author={Yong, Jiongmin and Zhou, Xun Yu},
  volume={43},
  year={1999},
  publisher={Springer Science \& Business Media}
}

@article{sethi2024modified,
  title={The modified MSA, a gradient flow and convergence},
  author={Sethi, Deven and {\v{S}}i{\v{s}}ka, David},
  journal={The Annals of Applied Probability},
  volume={34},
  number={5},
  pages={4455--4492},
  year={2024},
  publisher={Institute of Mathematical Statistics}
}

@article{kerimkulov2025mirror,
  title={Mirror descent for stochastic control problems with measure-valued controls},
  author={Kerimkulov, Bekzhan and {\v{S}}i{\v{s}}ka, David and Szpruch, {\L}ukasz and Zhang, Yufei},
  journal={Stochastic Processes and their Applications},
  pages={104765},
  year={2025},
  publisher={Elsevier}
}

@article{sethi2025mirror,
  title={Mirror descent for constrained stochastic control problems},
  author={Sethi, Deven and {\v{S}}i{\v{s}}ka, David},
  journal={arXiv preprint arXiv:2506.02564},
  year={2025}
}

@article{reisinger2023linear,
  title={Linear convergence of a policy gradient method for some finite horizon continuous time control problems},
  author={Reisinger, Christoph and Stockinger, Wolfgang and Zhang, Yufei},
  journal={SIAM Journal on Control and Optimization},
  volume={61},
  number={6},
  pages={3526--3558},
  year={2023},
  publisher={SIAM}
}

@article{porner2016iterative,
  title={An iterative Bregman regularization method for optimal control problems with inequality constraints},
  author={P{\"o}rner, Frank and Wachsmuth, Daniel},
  journal={Optimization},
  volume={65},
  number={12},
  pages={2195--2215},
  year={2016},
  publisher={Taylor \& Francis}
}

@article{porner2018inexact,
  title={Inexact iterative Bregman method for optimal control problems},
  author={P{\"o}rner, Frank},
  journal={Numerical Functional Analysis and Optimization},
  volume={39},
  number={4},
  pages={491--516},
  year={2018},
  publisher={Taylor \& Francis}
}

@article{chernousko1982method,
  title={Method of successive approximations for solution of optimal control problems},
  author={Chernousko, Felix L and Lyubushin, AA},
  journal={Optimal Control Applications and Methods},
  volume={3},
  number={2},
  pages={101--114},
  year={1982},
  publisher={Wiley Online Library}
}

@article{krylov1972algorithm,
  title={An algorithm for the method of successive approximations in optimal control problems},
  author={Krylov, Igor Anatol'evich and Chernous' ko, Feliks Leonidovich},
  journal={USSR Computational Mathematics and Mathematical Physics},
  volume={12},
  number={1},
  pages={15--38},
  year={1972},
  publisher={Elsevier}
}

@article{kerimkulov2021modified,
  title={A modified MSA for stochastic control problems},
  author={Kerimkulov, Bekzhan and {\v{S}}i{\v{s}}ka, David and Szpruch, Lukasz},
  journal={Applied Mathematics \& Optimization},
  volume={84},
  number={3},
  pages={3417--3436},
  year={2021},
  publisher={Springer}
}

@article{schulman2017proximal,
  title={Proximal policy optimization algorithms},
  author={Schulman, John and Wolski, Filip and Dhariwal, Prafulla and Radford, Alec and Klimov, Oleg},
  journal={arXiv preprint arXiv:1707.06347},
  year={2017}
}

@inproceedings{schulman2015trust,
  title={Trust region policy optimization},
  author={Schulman, John and Levine, Sergey and Abbeel, Pieter and Jordan, Michael and Moritz, Philipp},
  booktitle={International conference on machine learning},
  pages={1889--1897},
  year={2015},
  organization={PMLR}
}

@article{wagener2019online,
  title={An online learning approach to model predictive control},
  author={Wagener, Nolan and Cheng, Ching-An and Sacks, Jacob and Boots, Byron},
  journal={arXiv preprint arXiv:1902.08967},
  year={2019}
}

@article{sethi2025entropy,
  title={Entropy annealing for policy mirror descent in continuous time and space},
  author={Sethi, Deven and {\v{S}}i{\v{s}}ka, David and Zhang, Yufei},
  journal={SIAM Journal on Control and Optimization},
  volume={63},
  number={4},
  pages={3006--3041},
  year={2025},
  publisher={SIAM}
}

\end{document}